\documentclass{article}
\usepackage{amsfonts}
\usepackage{tikz-cd}
\usepackage[utf8]{inputenc}
\usepackage{color}

\newcommand{\blue}[1]{\textcolor{blue}{#1}}

\usepackage{mathrsfs}
\usepackage[OT2, T1]{fontenc}
\usepackage{url}
\usepackage{amsmath}
\usepackage{mathabx}
\usepackage{array}
\usepackage{graphicx}
\usepackage{amssymb}
\usepackage{amstext}
\usepackage{amsthm}
\usepackage{hyperref}
\usepackage{colonequals}
\usepackage{enumitem}
\usepackage[alphabetic,lite]{amsrefs}
\usepackage{cleveref}
\usepackage[all,cmtip]{xy}
\usepackage{fullpage}
\numberwithin{equation}{subsection}

\newtheorem{theorem}{Theorem}[section]
\newtheorem{lemma}[theorem]{Lemma}
\newtheorem{proposition}[theorem]{Proposition}

\newtheorem{corollary}[theorem]{Corollary}

\theoremstyle{definition}
\newtheorem{eg}[theorem]{Example}
\newtheorem{definition}[theorem]{Definition}

\theoremstyle{remark}
\newtheorem{remark}[theorem]{Remark}

\newtheorem*{claim}{Claim}

\DeclareMathOperator{\G}{\mathbb{G}}

\DeclareMathOperator{\bF}{\mathbb{F}}

\DeclareMathOperator{\GL}{GL}

\DeclareMathOperator{\Loc}{Loc}

\DeclareMathOperator{\Spec}{Spec}
\DeclareMathOperator{\End}{\mathsf{End}}

\DeclareMathOperator{\Oo}{\mathcal{O}}

\newcommand{\ra}{\rightarrow}

\newcommand{\A}{\mathbb{A}}

\newcommand{\bbC}{\mathbb{C}}
\newcommand{\D}{\mathbb{D}}

\newcommand{\F}{\mathbb{F}}

\newcommand{\Q}{\mathbb{Q}}
\newcommand{\ol}{\overline}

\newcommand{\R}{\mathbb{R}}

\newcommand{\V}{\mathbb{V}}
\newcommand{\Z}{\mathbb{Z}}
\newcommand{\bL}{\mathbb{L}}
\newcommand{\W}{\mathbb{W}}

\newcommand{\bbA}{\mathbf{A}}

\newcommand{\bbL}{\mathbf{L}}
\newcommand{\bbN}{\mathbf{N}}

\newcommand{\bbS}{\mathbf{S}}

\newcommand{\bbZ}{\mathbf{Z}}

\newcommand{\cA}{\mathcal{A}}

\newcommand{\cC}{\mathcal{C}}
\newcommand{\cD}{\mathcal{D}}
\newcommand{\cE}{\mathcal{E}}

\newcommand{\cH}{\mathcal{H}}

\newcommand{\cM}{\mathcal{M}}

\newcommand{\cO}{\mathcal{O}}

\newcommand{\cP}{\mathcal{P}}
\newcommand{\cQ}{\mathcal{Q}}

\newcommand{\cS}{\mathcal{S}}
\newcommand{\cT}{\mathcal{T}}
\newcommand{\cU}{\mathcal{U}}
\newcommand{\cV}{\mathcal{V}}
\newcommand{\cW}{\mathcal{W}}
\newcommand{\cX}{\mathcal{X}}
\newcommand{\cY}{\mathcal{Y}}
\newcommand{\cZ}{\mathcal{Z}}

\newcommand{\Ag}{\mathcal{A}_g}

\newcommand{\ShimK}{S_K(G,X)}
\newcommand{\Shim}{S(G,X)}

\newcommand{\integralShimK}{\mathscr{S}_K(G,X)}
\newcommand{\barintegralShimK}{\bar{\mathscr{S}}_K(G,X)}

\DeclareMathOperator{\MF}{MF}

\DeclareMathOperator{\opp}{opp}

\newcommand{\dR}{_{\mathrm{dR}}}
\newcommand{\et}{{\mathrm{et}}}
\newcommand{\cris}{_{\mathrm{cris}}}

\newcommand{\Fil}{{\mathrm{Fil}}}

\newcommand{\can}{{\mathrm{can}}}
\newcommand{\an}{{\mathrm{an}}}

\newcommand{\frob}{\varphi}

\DeclareMathOperator{\Sp}{Sp}

\DeclareMathOperator{\sss}{ss}
\DeclareMathOperator{\unip}{unip}

\DeclareMathOperator{\Isom}{Isom}
\DeclareMathOperator{\Mod}{Mod}
\DeclareMathOperator{\Rep}{Rep}

\DeclareMathOperator{\Gal}{Gal}
\DeclareMathOperator{\Hom}{Hom}
\DeclareMathOperator{\Aut}{Aut}
\DeclareMathOperator{\Lie}{Lie}

\DeclareMathOperator{\Spf}{Spf}
\DeclareMathOperator{\ad}{ad}

\DeclareMathOperator{\Id}{Id}

\DeclareMathOperator{\Gr}{Gr}

\newcommand{\Fildot}{\Fil^{\bullet}}

\def\codim{\operatorname{codim}}
\def\bfM{\mathbf{M}}
\def\bfN{\mathbf{N}}

\title{Integral canonical models of exceptional Shimura varieties}
\author{Benjamin Bakker, Ananth N. Shankar, Jacob Tsimerman}

\begin{document}
\maketitle

\begin{abstract}
    We prove that Shimura varieties admit integral canonical models for sufficiently large primes. In the case of  abelian-type Shimura varieties, this recovers work of Kisin-Kottwitz  for sufficiently large primes. We also prove the existence of integral canonical models for images of period maps corresponding to geometric families. We deduce several consequences from this, including a version of Tate semisimplicity, CM lifting theorems, and a version of Tate's isogeny theorem for ordinary points.

\end{abstract}

\setcounter{tocdepth}{2}
\tableofcontents

\section{Introduction}

This paper is motivated by the following question, raised by Langlands and made precise by Milne:
\vskip1em
\textit{Given a Shimura variety $\ShimK$ over its reflex field $E$, does $\ShimK$ admit a ``good model'' $\integralShimK$, over completions $\cO_{E_v}$?}
\vskip1em

For the moduli space of principally polarized abelian varieties $\Ag$, for instance, the moduli problem provides a natural model.  As not all Shimura varieties are known to have an associated geometric moduli problem (notably the exceptional Shimura varieties), it is preferable to have a model which is uniquely determined by certain intrinsic properties, possibly only for sufficiently large primes.  At the very least, we would like our model to be smooth over $\cO_{E_v}$, which already may not exist at certain primes.  Indeed, even in the case of the modular curve, the level structure at $p$ needs to be maximal in order for there to exist a smooth model at $p$. 

Smoothness alone however does not a priori guarantee a unique model. Milne defined an ``extension property'' (and this notion was refined by Moonen in \cite{Moonen}) which formally guarantees uniqueness. 
Following Milne and Moonen, an integral model is said to be an \emph{integral canonical model} if it is smooth and if it satisfies their extension property.  The moduli space of principally polarized abelian varieties $\Ag$, for instance, is known to admit integral canonical models over $\Z_p$ whenever the level structure is prime to $p$, i.e. when the level structure is hyperspecial at $p$.  Indeed, smoothness follows from the fact that the deformation theory of polarized abelian varieties with a prime-to-$p$-polarization is unobstructed (proved by Grothendieck), and Chai and Faltings \cite{faltingschai} prove an extension result for abelian schemes over certain bases (see \cite[Section 3]{Moonen} for the precise results they prove) from which the Milne--Moonen extension property for $\Ag$ follows directly.  

Langlands suggested more generally that integral canonical models should exist whenever the compact open $K_p$ at $p$ is hyperspecial.  In the case of PEL Shimura varieties (i.e. Shimura varieties parameterizing polarized abelian varieties with extra endomorphisms), Kottwitz in \cite{Ko} constructs integral models when the level structure at $p$ is hyperspecial. Kisin generalizes Kottwitz' work to the setting of Shimura varieties of abelian type (\cite{Kisinintegral}). In \cite{imai2023prismatic} the authors prove a prismatic characterization of integral models of abelian type. In \cite{PappasRapoport}, Pappas and Rapoport lay out a framework for integral models for all Shimura varieties, and amongst various other things, they show that the integral models constructed by Kisin satisfy their framework. In \cite{Lovering}, Lovering constructs integral canonical models for automorphic data associated to Shimura varieties of abelian type.  The extension property proven in most of these works ultimately relies on the work of Chai and Faltings, which uses the moduli intepretation of $\Ag$, and therefore does not carry over to exceptional Shimura varieties.

\subsection{The extension property}\label{intro:extensionproperty}
We now define our extension property. We note that our extension property is the equivalent to the one formulated by Milne-Moonen, and we include a discussion about this fact below. In the proper case, the extension property is very simple: 
\begin{definition}\label{properextensionproperty}
    Let $F$ be a local field. We say that an integral model $\integralShimK /\cO_F$ of a proper Shimura variety $\ShimK/F$ has the extension property if for every smooth scheme $\cT/\cO_F$, every map over the generic fiber $\phi\in\Hom(\cT_F,\ShimK)$ extends to a map from $\cT$ to $\integralShimK$.
\end{definition}

In the non-proper case, one must take care to ensure that the special fiber doesn't map to the ``boundary'' of $\integralShimK$. For primes $\ell \neq p$, we fix an $\ell$-adic automorphic local system $_{\et}\V_{\ell}$ on $\integralShimK$ corresponding to a faithful representation $V$ of $G^{\ad}$ (see Section \ref{sec: Shimura setup} for precise definitions of $_{\et}\V_{\ell}$).  
\begin{definition}\label{generalextensionproperty}
We say that a model $\integralShimK/\cO_F$ of $\ShimK$ satisfies the extension property if 
for every smooth scheme $\cT/\cO_F$, and every map over the generic fiber $\phi\in\Hom(\cT_F, \ShimK)$ such that $\phi^*\V_{\ell}$ extends to $\cT$ for some $\ell\neq p$, the map $\phi$ extends to a map from $\cT$ to $\integralShimK$.
    
\end{definition}

\subsection{Main results for Shimura varieties}

Our main theorem is to prove the existence of integral canonical models for arbitrary Shimura varieties $S$, for sufficiently large primes. Concretely, we prove the following: 

\begin{theorem}\label{shimura main}
    Let $S$ be a Shimura variety over its reflex field $E$. For almost all finite places $v$, there exists a model $\cS$ over $\cO_v$ which has the extension property, and which admits a log-smooth compactification over $\cO_v$. 

    Moreover, for the places $v$ as above, $\cS$ is the unique smooth model of $S$ over $\cO_v$ which admits a log-smooth compactification over $\cO_v$ and to which the local system $_{\et}\V_{\ell}$ extends. 
\end{theorem}

Note that the uniqueness claim implies that any global model of $\ShimK$ will specialize to the integral canonical model at almost all places. 

While our theorem applies to all Shimura varieties, the fact that it only applies for almost all finite places means that much stronger results (\cite{Kisinintegral}) are known in the abelian setting. Therefore, the 
case of greatest interest is when $S$ is a Shimura variety of exceptional type. 

\subsubsection{Applications}
We deduce several applications for mod $v$ (and $v$-adic) Shimura varieties. The main theorem implies that Hecke correspondences away from $p$ extend to $\integralShimK/\cO_{E_v}$. We use these Hecke correspondences to prove the following theorem: 

\begin{theorem}\label{Tatessintro}
    Let $x\in \integralShimK(\F_q)$ be any point. Then, the Frobenius endomorphism on $_{\et}\V_{\ell,\bar{x}}$ is semi-simple. 
\end{theorem}
Work of Esnault-Groechenig shows that $\V_{p}/\ShimK$ is a crystalline local system, which yields a relative Faltings-Fontaine-Laffaille on the $p$-adic completion of $\integralShimK$. Analogous to \cite[Section 1.5]{Kisinintegral}, we give an explicit description of the this Faltings-Fontaine-Laffaille module restricted to complete local rings of $\integralShimK$ at closed points (see Theorem \ref{thm: Kisin deformation} for the precise statement). We then use this in conjunction with the semi-simplicity of $\ell$-adic Frobenii, to prove the following theorem (which generalizes the Serre-Tate canonical lift).

\begin{theorem}\label{CMliftsintro}
The $\mu$-ordinary locus in $\integralShimK_{\F_v}$ is open dense. 
    Let $x\in \integralShimK(\F_q)$ denote a $\mu$-ordinary point. Then, there exists a canonical special point $\tilde{x}\in \integralShimK(W(\F_q))$ that lifts $x$. 
\end{theorem}

The above theorem then immediately has the following consequence. 

\begin{theorem}\label{crysss}
    Let $x\in \integralShimK(\F_q)$ denote a $\mu$-ordinary point, and let $_{\cris}\V_x$ denote the $F$-crystal at $x$ (associated to the Fontaine-Laffaille module defined by $_{\et}\V_p/\ShimK$). Then, the crystalline $q$-Frobenius endomorphism on $_{\cris}\V_x$ is semisimple. 
\end{theorem}

Finally, we also have an analogue of Tate's isogeny theorem for ordinary points. 
\begin{theorem}\label{Tateisog}
        Let $x$ and $y$ be ordinary mod $p$ points whose $\ell$-adic Galois representations are isomorphic. Then a rational Hodge structures underlying the canonical lifts $\tilde{x}$ and $\tilde{y}$ are isomorphic. 
\end{theorem}
This result should be thought of as an exceptional analogue (in the ordinary case) of Tate's isogeny theorem, which states that two abelian varieties over finite fields are isogenous if the $\ell$-adic Galois representations associated to the abelian varieties are isomorphic. 

\subsection{Comparison with the extension property of Milne--Moonen}

The usual (Milne-Moonen) phrasing of the extension property is to let the level $K$ be of the form $K=K^pK_p$ where $K_p\subset G(\Q_p)$ and $K^p\subset G(\bbA^p_f)$, and consider the tower of varieties $\integralShimK$ where $K_p$ is fixed and $K^p$ varies. One then asks for a tower of integral models with finite etale transition maps, and that a map from the generic fiber of any scheme $\cT$ to this tower extends. This however entails defining a category of \textit{admissible test schemes} to be able to map to this entire tower, which is something like ``pro-smooth schemes over $\bbZ_p$'', and this requires some care to do precisely. Geometrically these are more complicated (pro-) objects, which is why we choose to work with our definition.

These two extension properties are very close to each other. Indeed, the Milne-Moonen extension property implies ours as follows: Suppose $\cT$ is a smooth scheme over $\cO_v$, and $\phi:\cT_{F_v}\ra \ShimK$ such that all of the $\phi^*_{\et}\V_\ell$ extend. Then we may define a tower $\cT_{K^p}$ with finite transition maps using the local systems $\phi^*_{\et}\V_\ell$, and a corresponding map from $\varprojlim_{K^p}Y_{K^p,F_v}\ra \varprojlim_{K^p}\integralShimK_{F_v}$. The Milne-Moonen extension property implies that this map extends and hence so does our original one.

Conversely, we first point out that our construction in fact gives models for the entire tower $\varprojlim_{K_p} \integralShimK_{F_v}$. Suppose one has an admissible test $\cT$ over $\cO_v$ and a map from $\cT_{F_v}$ to $\varprojlim_{K_p} \integralShimK_{F_v}$. By definition, $\cT$ is a pro-etale cover of a smooth scheme over some unramified extension $R$ of $\cO_v$, and hence is an inverse limit of smooth schemes over $R$. Each of the schemes in this inverse limit must therefore map to some finite level Shimura variety in the tower, in a way that limits to the original map. Moreover, the pullbacks of all the automorphic local systems $_{\et}\V_\ell$ will extend, and thus our extension property will apply at each finite level, and hence to all of $\cT$. 

\subsection{Integral canonical models for images of geometric period maps}
Our methods in fact provide integral canonical models at sufficiently large primes for many varieties which are interpretable as moduli spaces of Hodge structures.  We describe the case of variations coming from geometry, but first give some general background from Hodge theory (see \Cref{hodge sect} for more details).

Suppose an irreducible normal complex variety $P$ (which we can without loss of generality assume smooth) is equipped with a polarizable integral variation of Hodge structures $(\W_\Z,F^\bullet)$.  The associated period map $P^\an\to [\mathbf{G}(\Z)\backslash D]$ will factor as $P^\an\to Y^\an\to[\mathbf{G}(\Z)\backslash D]$ where $Y$ is a generically inertia-free Deligne--Mumford stack with quasiprojective coarse moduli space, $P\to Y$, has connected generic geometric fiber, and $Y^\an\to[\mathbf{G}(\Z)\backslash D]$ is finite \cite{bbt}.  This map is obtained by taking a resolution $\tilde P\to P$ and then a partial compactification $\tilde P\subset P'$ for which the period map extends to a proper period map and forming the Stein factorization.  In particular, the variation $(\W_\Z,F^\bullet)$ is pulled back from $(\V_\Z,F^\bullet)$ on $Y$.  In general $\tilde P$ may be a stack as well, but the inertia of $Y$ acts faithfully on $\V_\Z$, so $Y$ is a global quotient of a normal quasiprojective variety by the action of a finite group, and after passing to a finite \'etale cover of $P$ (for instance by adjoining level structure) we may assume $Y$ (and therefore $P$) is a quasiprojective variety. 
 We call such a $Y$ the Stein factorization of the period map.  
 
Now let $P$ be an irreducible normal variety defined over a number field $E$, $f:Z\to P$ a smooth projective $E$-morphism, and for any fixed $k$ let $_{et}\W_\ell:=R^kf_*\Z_\ell$ be the arithmetic local system---or more generally a subquotient obtained from maps between such local systems induced by $P$-morphisms of families\footnote{We lay out the precise conditions for a variation of Hodge structures to ``come from geometry'' in the sense we need it in \Cref{localcanonicalmodels}.}.  We let $Y$ be the Stein factorization of the period map of the variation with underlying geometric local system $\V_\Z=R^kf^\an_*\Z$ for some embedding $E\subset \bbC$.  Then $Y$ is defined over $E$, does not depend on the embedding, and the local system $_{et}\W_\ell$ descends to $Y$.

For example, for any moduli space $P$ admitting an infinitesimal Torelli theorem (including moduli spaces of Calabi--Yau varieties and most moduli spaces of complete intersections in $\mathbb{P}^n$, to name a couple), $Y$ is a canonical partial compactification (namely, the natural ``Hodge--theoretic'' compactification) of $P$.
\begin{theorem}\label{thm ICM hodge}
      Let $E$ be a number field.  Let $Y/E$ be the Stein factorization of the period map of a geometric variation of Hodge structures which we assume to be inertia-free (and therefore a quasiprojective variety).  Let $_{et}\V_\ell$ be the descent of the natural $\ell$-adic arithmetic local system to $Y$.
    \begin{enumerate}
    \item\label{thm ICM hodge p1} There is a finite set of primes $\Sigma_Y$ of $\cO_E$ such that, for all finite places $v$ of $\cO_{E,\Sigma_Y}$, there is a model $\cY_v$ over $\cO_v$ such that:
    \begin{enumerate}
    \item\label{thm ICM hodge p1a} Each $_{et}\V_\ell$ extends to $\cY_v$ for $\ell\neq p$, where $p$ is the integral prime below $v$.
    \item\label{thm ICM hodge p1b} For any irreducible smooth $\cX/\cO_v$ with generic point $\eta_\cX$, any morphism $g_0:\eta_\cX\to Y_{E_v}$ extends to $g:\cX\to \cY_v$ if and only if $g_0^*(_{et}\V_\ell)$ extends to $\cX$ for some $\ell\neq p$.
    \end{enumerate}
        \item\label{thm ICM hodge p2} The models in part \ref{thm ICM hodge p1} are functorial in the following sense.  Let $Y,Y'/E$ be the Stein factorizations of two possibly different period maps arising as above, and for a finite place $v$ of $\cO_{E,\Sigma_Y\cup \Sigma_{Y'}}$ let $\cY_v,\cY'_v$ be models as in part \ref{thm ICM hodge p1}.  Then any morphism $g_0:\eta_\cY\to \cY'_v$ extends to $g:\cY_v\to \cY'_v$ if and only if $g_0^*(_{et}\V_\ell)$ extends to $\cY_{v}$ for some $\ell\neq p$.  In particular, the local model $\cY_v$ of $Y$ is canonically determined.
        \item\label{thm ICM hodge p3} The models in part \ref{thm ICM hodge p1} are minimal in the following sense.  For any $\cO_v$-model $\cZ_v$ of $Y$ that admits a uniform log smooth resolution (see \Cref{defn uniform res}) and to which $_{et}\V_\ell$ extends for some $\ell\neq p$, the identity map over the generic fiber extends to an $\cO_v$-morphism $\cZ_v\to\cY_v$.  
\item\label{thm ICM hodge p4} Finally, there exists a global model $\cY$ over $\cO_{E,\Sigma_Y}$ of $Y$ which restricts to the model $\cY_v$ of part \ref{thm ICM hodge p1} for all places $v$ of $\cO_{E,\Sigma_Y}$.

    \end{enumerate}
\end{theorem}

In general, the image of a period map has no reason to be smooth (already over $\bbC$), so it is important to allow for singularities.  Even starting with a smooth moduli space $S$ parametrizing varieties with an infinitesimal Torelli theorem so that the period map $S^\an\to[ \mathbf{G}(\Z)\backslash D]$ is immersive, it will often not be proper; the ``Hodge-theoretic'' compactification is essentially the closure of the image of $S^\an\to [\mathbf{G}(\Z)\backslash D]$ (up to a finite map) and will therefore usually be singular anyway.  This is already true for the image of the Torelli map for the moduli space of curves (see \Cref{eg mg ag}).

\subsection{An outline of the proof of the Extension Property}
The proof of the extension property in \Cref{shimura main} is motivated by classical Hodge theory and the arguments of \cite{dejongoort,moretbailly,faltingschai}.  Briefly, let $F/\Q_p$ be a finite unramified extension ring of integers $\cO$, $\cX/\cO$ a smooth scheme, and $f_F:X\to Y$ a morphism over the generic point to a model of a Stein factorization of a period map $\cY$.  There are three steps to proving the extension of $f_F$ to a morphism $f:\cX\to\cY$:
\begin{enumerate}
    \item $\cY$ will support an $\ell$-adic local system $_{et}\V_\ell$ with large local monodromy at the boundary. 
 It then follows (assuming the existence of a reasonable resolution and log smooth compactification) that the map extends in codimension one as soon as $f_F^*(_{et}\V_\ell)$ extends to $\cX$.  This in particular establishes the claim if $\dim \cX=1$.
    \item   If $\dim\cX>2$, then assuming the result for $\dim\cX=2$, the conclusion follows by induction.  An extension is unique if it exists, and by assumption exists on a general hyperplane section $\cH$.  Deformation theory can then be used to locally lift the extension to the completion of $\cX$ along $\cH$, which is enough to deduce the existence of the extension to $\cX$.  This is the same argument used by \cite{dejongoort,moretbailly,faltingschai}.
    \item The main step is then to handle the case $\dim\cX=2$, so when $\cX/\cO$ is a family of smooth curves, and here we must use that $\cY$ supports\footnote{This only makes sense after pulling back to a resolution which is smooth over $\cO$} a crystalline local system whose associated Fontaine--Laffaille module has ample Griffiths bundle.  On the one hand, by a key result of Hokaj \cite{Hokaj} and Guo--Yang \cite{guo}, the Fontaine--Laffaille module extends to $\cX/\cO$; on the other hand, the map to $\cY$ can be extended on a blow-up $\cX'$ at points in the special fiber.  The pull-backs of the two Fontaine--Laffaille modules must agree, and the ampleness of the Griffiths bundle implies every exceptional curve is contracted in $\cY$.

\end{enumerate}

\subsection{Comparison with past work}
In the Abelian case which Kisin handles, proving the existence of smooth models is difficult, whereas the extension property follows directly from the extension property for $\Ag$. Unlike \cite{Kisinintegral}, we don't directly address the question of for which places $v$ the Shimura variety $\ShimK$ admits smooth models over $\cO_v$. Instead, we ``spread out'' the global model $\ShimK$ over $\cO_E[1/N]$ for some large $N$, which in particular guarantees smoothness at primes $v\nmid N$ (amongst other good properties). However, the lack of an interpretation of our Shimura varieties as a moduli space means we need an entirely different argument for establishing the extension property.

Regarding the applications to the mod $p$ and $p$-adic geometry of Shimura varieties, Kisin (\cite{Kisinmodppoints}) proves that every $\overline{\bF}_p$-point of abelian Shimura varieties admit CM lifts up to isogeny. Note that the up-to-isogeny requirement is crucial as work of Oort \cite{oortJAG} shows (for a more general statement, see also \cite{KLSS}). Noot in \cite{Noot} proves the existence of CM lifts (on the nose) for points of Hodge-type Shimura varieties that intersect the ordinary locus of $\Ag$. Noot's theorem was generalized to $\mu$-ordinary points by \cite{MoonenSerreTate} in the PEL-type case to the  and \cite{ShankarZhou} in the Hodge type case.

\subsection{Acknowledgements}
B. B. benefited from many conversations with Yohan Brunebarbe throughout, and first learned of the work of Ogus--Vologodsky from Raju Krishnamoorthy. We are also grateful to Michael Groechenig, Kiran Kedlaya, Mark Kisin, Arthur Ogus, Abhishek Oswal, and Vadim Vologodsky for helpful conversations.  We are particularly indebted to Alex Youcis for pointing out the result of Hokaj \cite{Hokaj}, which enabled us to correct an error in the original version.  B. B. was partially supported by NSF DMS-2131688. A. S. was partially supported by the NSF grant DMS-2100436 and DMS-2338942, and would like thank MSRI for their hospitality during the spring of 2023. 

\subsection{Outline}
In \Cref{sec: boundary and monodromy} we formulate the definition of an integral canonical model in terms of the extension property, and show that extension in codimension 1 can be controlled in many cases (including those of interest) by a local system.  We also discuss the descent and extension of arithmetic local systems to integral models of period images, and the relation of the extension property to purity results in classical Hodge theory.  In  \Cref{sect:codim 2} we prove the extension in codimension 2.  In \Cref{localcanonicalmodels} we formulate precise conditions under which the presence of a Fontaine--Laffaille module on a local model implies the extension property, and in \Cref{sec: canonical models} we apply it to the case of Shimura varieties and period images to prove the existence of integral canonical models as in \Cref{shimura main} and \Cref{thm ICM hodge}. 
 In \Cref{sec: Tate semisimplicity} we prove Hecke correspondences extend and use this to prove \Cref{Tatessintro}.   In Sections \Ref{sec: complete local rings for large primes} and \ref{sec: CM stuff}, we analyze the complete local rings of the integral canonical models of Shimura varieties at closed points and prove Theorems \ref{CMliftsintro}, \ref{crysss}, and \ref{Tateisog}.

\subsection{Notation}Typically $\cO$ will be the ring of integers in either a number field $E$ or an unramified finite extension $F/\Q_p$.  We use script letters $\cS,\cX,\cY$ to denote either finite type $\Spec \cO$-schemes or finite type formal $\Spf\cO$-scheme (in the latter case).  We denote by unadorned Roman letters $S,X,Y$ the generic fibers over $E$ or $F$ (in the non-formal case) and $\cS_0,\cX_0,\cY_0$ will denote the special fibers (when $\cO$ is local).  For a finite place $v$ of $\cO_E$ we denote by $\cO_v=\cO_{E_v}$ the completion.

\def\piet{\pi_1^{\et}}

\section{The extension property and natural boundary}\label{sec: boundary and monodromy}
In this section we single out a class of integral models for which extension on the level of fundamental groups is the only obstruction to extending morphisms defined over the generic point.  In particular, such models will be uniquely determined when they exist.  The key property these models will have is the extension of morphisms from smooth schemes, just as in the work of Milne and Moonen.  In the Shimura case (or whenever the model is smooth) this property obviously uniquely determines the model; in the general case it will be important to allow for singularities, as even for smooth moduli spaces, the extension property will often force a partial compactification which is not smooth---for instance, see \Cref{eg mg ag}.  We also discuss how the extendability on the level of fundamental groups can be detected by the extendability of local systems.

 
Throughout this section, we let $\cO$ be a Dedekind domain with fraction field $F$ which is either a finite extension of $\Q_p$ or a global number field.  In the local situation, we denote by $k$ the residue field.

\subsection{The extension property}

\begin{definition}\label{smooth extension prop}\hspace{.5in}
    \begin{enumerate}
    \item Let $\cY,\cY'/\cO$ be finite type flat $\cO$-schemes.  An $F$-morphism $f_F:Y\to Y'$ of the generic fibers is $\cO$-admissible if the maps on fundamental groups fits into a diagram
        \[\begin{tikzcd}
            \piet(Y,\bar y)\ar[d]\ar[r]&\piet(Y',\bar y')\ar[d]\\
            \piet(\cY,\bar y)\ar[r,dashed]&\piet(\cY',\bar y')
        \end{tikzcd}\]
        for compatibly chosen basepoints.
 
        \item    We say a finite type flat $\cO$-scheme $\cY/\cO$ satisfies the smooth extension property over $\cO$ if for any smooth $\cO$-scheme $\cX/\cO$, any $\cO$-admissible $F$-morphism $f_F:X\to Y$ on the generic fiber extends to an $\cO$-morphism $f:\cX\to\cY$.
    \end{enumerate}
 
\end{definition}

\subsection{Natural boundary and extension in codimension 1}
In the presence of a sufficiently well-behaved log smooth compactification of a resolution (one with ``natural boundary'' as defined below), admissible morphisms with smooth source will extend over codimension 1 points, and this will be the first step of the proof of the extension property in the setting of section \ref{localcanonicalmodels}.

The following definition is made with a view towards the general case where $\cY$ is not smooth, but for Shimura varieties we will have $\cS=\cY$. 

\begin{definition}\label{defn nat bound}\hspace{.5in} 
\begin{enumerate} 
    \item Let $\cY/\cO$ be a normal finite type flat $\cO$-scheme.  A resolution of $\cY/\cO$ is a smooth $\cO$-scheme $\cS/\cO$ with a proper dominant birational map $\pi:\cS\to\cY$ such that $\cO_\cY\xrightarrow{\cong}\pi_*\cO_\cS$ via the natural map.
    \item We say $(\bar\cS,\cD)/\cO$ is log smooth if it is \'etale locally isomorphic over $\cO$ to $(\mathbb{A}^n_\cO,\cD)$, where $\cD$ is a union of coordinate hyperplanes.
    \item Let $\cS/\cO$ be a smooth $\cO$-scheme.  By a log smooth compactification we mean a log smooth proper $\cO$-scheme $(\bar \cS,\cD)/\cO$ with $\bar\cS\setminus \cD=\cS$ (as an $\cO$-scheme).

\end{enumerate}
\end{definition}
In the cases that will be of interest to us, the boundary of $\bar\cS$ will be detected by the fundamental group in the sense made precise by the following definition.  For a log smooth scheme $(\bar S,D)$ over $\bbC$ with $S=\bar S\setminus D$ and whose boundary $D$ has irreducible components $\{D_i\}_{i \in I}$, and for any point $x\in D$ together with a small arc $\hat x:[0,1)\to \bar S$ mapping $0$ to $x$ and $(0,1)$ to $S$, the local fundamental group $\pi_1^\mathrm{loc}(\bar S,D,\hat x):=\lim_{x\in U\subset\bar S}\pi_1(U\setminus D,\hat x)$ is naturally identified with $\bigoplus_{D_i\ni x}\Z$, where the sum is over irreducible component $D_i$ of $D$ containing $x$ and the basepoint is taken to be any point on $\hat x$ sufficiently close to $x$.  
\begin{definition}
Let $(\bar S,D)$ be a proper log smooth complex variety with $S=\bar S\setminus D$ and $\pi:S\to Y$ a resolution of a normal complex variety $Y$.  We say the boundary of $\bar S$ is $\pi$-natural if for every morphism $f:(\bar C,D_C)\to (\bar S,D)$ from a log smooth proper curve $(\bar C,D_C)$ with $D_C=(f^*D)^\mathrm{red}$ and for any $c\in D_C$, we have that the composition 
\[\pi_1^\mathrm{loc}(C,D_C,\hat c)\to \pi_1^\mathrm{loc}(\bar S,D,f(c))\to\pi_1(S,f(\hat c))\to \pi_1(Y,\pi\circ f(\hat c))\to\pi_1^\et(Y,\pi\circ f(\hat c))\]
is injective.  Equivalently, we require that for any $f,c$ as above and any $n$, there is a Galois finite \'etale cover $Y'\to Y$ whose base-change to $C$ ramifies to order $\geq n$ at $c$. 

Let $\cY/\cO$ be a normal finite type flat $\cO$-scheme, $\pi:\cS\to\cY$ a resolution, and $\bar \cS/\cO$ a log smooth compactification of $\cS$.  We say that the boundary of $\bar\cS$ is $\pi$-natural if this is true for the base-change of $\pi$ and $\bar \cS$ to $\bbC$ for some (hence any, by the above remark) embedding of $K$ into $\bbC$.  If $\pi$ is the identity we just say $\bar\cS$ has natural boundary.
    
\end{definition}

\begin{lemma}\label{extend codim 1}Let $F/\Q_p$ be an unramified extension  with ring of integers $\cO$.  Let $\cY/\cO$ be a normal finite type flat $\cO$-scheme, $\pi:\cS\to\cY$ a resolution, and $\bar \cS/\cO$ a log smooth compactification of $\cS$ with $\pi$-natural boundary.  Let $\cX/\cO$ be a normal finite type flat $\cO$-scheme and $f_F:X\to Y$ an $\cO$-admissible morphism over the generic point.  Then there is a finite type flat $\cX'/\cO$, a proper dominant generically finite morphism $\cX'\to\cX$, and a morphism $\cX'\to\cS$ such that the following diagram on the generic fibers commutes:
\[\begin{tikzcd}
    X'\ar[r]\ar[d]&S\ar[d,"\pi_F"]\\
    X\ar[r,"f_F"]&Y.
\end{tikzcd}\]
 Moreover, if the image of $f_F$ meets the locus of $Y$ where $\pi$ is an isomorphism, then $\cX'\to\cX$ can be taken to be an isomorphism over a dense open set.
\end{lemma}
\begin{proof}Let $\cW\subset \cX\times_\cO\bar\cS$ be the closure of the fiber product $X\times_Y S\subset \cX\times_\cO\bar\cS$.  We then have that the projection $\cW\to\cX$ is proper dominant with a morphism $\cW\to \bar\cS$, and taking an appropriate complete intersection of ample hypersurfaces, we may find $\cX'\subset\cW$ which is proper dominant generically finite over $\cX$.  If $X$ meets the locus where $\pi$ is an isomorphism, then we take $\cX'=\cW$ and the morphism $\cX'\to \cX$ is a generic isomorphism.  Either way, if the boundary of $\bar \cS$ intersects $\cX'$ nontrivially, then there is a finite extension $\cO'$ of $\cO$ and an $\cO'$-point of $\cX'/\cO$ whose closed point maps to the boundary.  Because the boundary is natural, by the argument of \cite[Thm 4.4]{PST} this contradicts the admissibility assumption.

\end{proof}
\begin{corollary}\label{extend codim 1 nat}
    In the setting of the lemma, assuming the image of $f_F$ meets the locus of $Y$ where $\pi$ is an isomorphism, there is a codimension $\geq 2$ subset $\cZ\subset\cX$ and an extension of $f_F$ to a morphism $f:\cX\setminus\cZ\to\cY$.
\end{corollary}
\begin{proof}
    $\cX'\to\cX$ is proper so every codimension 1 point $\eta$ lifts to $\cX'$ and maps to $\cS$, so $f_F$ extends over $\eta$.
\end{proof}

\subsection{Natural boundary and local systems}
For $\cS/\cO$ smooth and $\bar \cS$ a log smooth compactification, both the naturality of the boundary and the admissibility condition on $\cO$-morphisms to $\cS$ can be checked using an $\ell$-adic local system $\V_\ell$ on $\cS$.  The key property such a local system must have is ``nonextendability'' over the generic fiber, meaning it does not extend to any nontrivial partial compactification---see \cite[\S3]{brunebarbe}, where it is referred to as ``maximality''.

\begin{definition}
        Let $Y/\bbC$ be a normal variety with an $\ell$-adic local system $\mathbb{V}_\ell$.  We say $\mathbb{V}_\ell$ is nonextendable if for any resolution $\pi:S\to Y$, finite \'etale $p:S'\to S$, and partial log smooth compactification $ S'\subset \bar S'$, if the local system $(p\circ\pi)^*\mathbb{V}_\ell$ extends to $\bar S'$, then $S'=\bar S'$.
        
        For $Y/F$ a normal variety with an $\ell$-adic local system $\mathbb{V}_\ell$, we say $\mathbb{V}_\ell$ is geometrically nonextenable if the base-change to $Y_\bbC$ is, for some (hence any) complex embedding $F\subset\bbC$.  We say an $\ell$-adic local system on an integral model $\cY$ of $Y$ is geometrically nonextendable if the restriction to the generic fiber is.
\end{definition}

Nonextendability restricts well:

\begin{lemma}[{\cite[\S3]{brunebarbe}}]\label{faithful local monodromy}
    Let $Y/\bbC$ be a normal variety and $\V_\ell$ an $\ell$-adic local system on $Y$.  Then the following are equivalent:
    \begin{enumerate}
        \item $\V_\ell$ is nonextendable.
        \item For any curve $C/\bbC$ and any proper morphism $f:C\to Y$, $f^*\V_\ell$ is nonextendable.
        \item For any normal variety $X/\bbC$ and any proper morphism $f:X\to Y$, $f^*\V_\ell$ is nonextendable.
    \end{enumerate}
\end{lemma}

The next lemma interprets the naturality of the boundary and admissibility in terms of the local system as promised. 

\begin{lemma}\label{faithful => natural} Suppose $\cY/\cO$ is a normal finite type flat $\cO$-scheme, $\pi:\cS\to\cY$ a resolution, $\bar \cS/\cO$ a log smooth compactification of $\cS$, and $\mathbb{V}_\ell$ a geometrically nonextendable $\ell$-adic local system on $\cY$.  Then
\begin{enumerate}
    \item\label{local system 1} $\bar \cS$ has $\pi$-natural boundary.
    \item\label{local system 2} For $\cX/\cO$ regular, a morphism $f_F:X\to Y$ over the generic point which meets the locus of $Y$ where $\pi$ is an isomorphism is $\cO$-admissible if and only if $f_F^*\mathbb{V}_\ell$ extends to $\cX$.
\end{enumerate}
    
\end{lemma}
\begin{proof}
    For \ref{local system 1}, by \Cref{faithful local monodromy}, for any proper morphism $f:C\to S_\bbC$ from a curve, the local monodromy of $f^*\mathbb{V}_\ell$ around any boundary point is infinite.

    For \ref{local system 2}, the forward implication is obvious.  For the reverse implication, the same argument as in \Cref{extend codim 1} shows that $f_F$ extends to $f:\cX\setminus\cZ\to\cY$ for a codimension $\geq 2$ subset $\cZ$, but by purity of the branch locus $\piet(\cX\setminus\cZ,\bar x)\to\piet(\cX,\bar x)$ is an isomorphism.
\end{proof}
\begin{corollary}\label{local system extend}
    For $\cY,\V_\ell$ as in the lemma, $\cX/\cO$ a normal finite type flat $\cO$-scheme, and $f_F:X\to Y$ a morphism over the generic point such that $f_F^*\V_\ell$ extends to $\cX$. 
 Then there is a codimension $\geq 2$ subset $\cZ\subset\cX$ and an extension of $f_F$ to a morphism $f:\cX\setminus\cZ\to\cY$.
\end{corollary}

\subsection{Integral canonical models of smooth varieties}
For simplicity, we first define integral canonical models in the smooth case (to be generalized in the next section).
 \begin{definition}\label{defn smooth ICM}
     Let $\cY/\cO$ a quasiprojective smooth $\cO$-scheme.  We say $\cY/\cO$ is an integral canonical model if for every closed point $v$ of $\Spec\cO$, the base-change $\cY_v/\cO_v$ satisfies the extension property over $\cO_v$.
 \end{definition}
 The following is immediate:
 \begin{lemma}\label{smooth extend=>general}
    Let $\cS,\cS'/\cO$ be two smooth integral canonical models and $f_{F}:S\to S'$ a morphism over the generic point.
    \begin{enumerate}
        \item $f_F$ extends to a morphism $f:\cS\to\cS'$ if and only if there is a commutative diagram
            \[\begin{tikzcd}
            \piet(S,\bar s)\ar[d]\ar[r]&\piet(S',\bar s')\ar[d]\\
            \piet(\cS,\bar s)\ar[r,dashed]&\piet(\cS',\bar s')
        \end{tikzcd}\]
    for compatibly chosen basepoints.
    \item Suppose further that $\cS'$ admits a log smooth compactification and a geometrically nonextendable $\ell$-adic local system $\V_\ell$.  Then $f_F$ extends to a morphism $f:\cS\to\cS'$ if and only if $f_F^*(\V_\ell)$ extends to $\cS$.
    \end{enumerate}

 \end{lemma}
\subsection{Integral canonical models of normal varieties}
We begin with a simple example to demonstrate that, even for a smooth variety, a model which satisfies the smooth extension property might necessarily involve a partial compactification with singularities.

 \begin{eg}\label{eg mg ag}
    Consider the moduli stacks $\mathcal{A}_g$ and $\mathcal{M}_g$ of principally polarized $g$-dimensional abelian varieties and genus $g$ curves, say over $\mathbb{Z}_p$ for $p\gg0$, and let $\cM_g\to\cA_g$ be the Torelli map.  After taking sufficient level structure, both of these stacks are schemes, and the Torelli map is generically immersive.  $\mathcal{A}_g$ satisfies the smooth extension property, so the closure of the image of $\cM_g$ in $\cA_g$ will as well.  This closure is quite singular however:  the map $\cM_g\to\cA_g$ extends to the partial compactification $\tilde \cM_g\to\cA_g$ parametrizing compact-type nodal curves (that is, nodal curves whose dual graph is a tree), but the boundary gets contracted with positive-dimensional (non-uniruled) fibers, as the Jacobian does not record the attaching nodes.
\end{eg}

In the normal case, to deduce the extension of arbitrary morphisms from the case of morphisms with smooth source, we will require the existence of an integral resolution with the property that the subvarieties contracted in a special fiber are limits of subvarieties contracted in the generic fiber.
\begin{definition}\label{defn uniform res}
\def\Hilb{\operatorname{Hilb}}

Let $\cY/\cO$ be a normal finite type flat $\cO$-scheme.  
\begin{enumerate}
    \item A resolution of $\cY/\cO$ is a smooth $\cO$-scheme $\cS/\cO$ with a proper dominant map $\pi:\cS\to\cY$ such that $\cO_\cY\xrightarrow{\cong}\pi_*\cO_\cS$ via the natural map.
    \item A resolution $\pi:\cS\to\cY$ is uniform if there is a flattening stratification of $\pi$ by locally closed subsets $\cY^j$ of $\cY$, each of which is flat over $\Spec \cO$.
\end{enumerate}

\end{definition}

 \begin{lemma}\label{generic uniform}
     Let $Y/F$ be a normal variety, $\pi_F:S\to Y$ a resolution, and $\pi:\cS\to\cY$ a model of $\pi_F$ defined over $\cO$.  Then $\pi$ is a uniform resolution over a (nonempty) open set of $\cO$.
 \end{lemma}
\begin{proof}
    Since $\cO_\cY\to\pi_*\cO_\cS$ is an isomorphism over the generic point, it is over an open set of $\Spec\cO_F[1/N]$.  Taking a flattening stratification $\cY^j$, the strata which do not dominate map to finitely many primes.
\end{proof}

 \begin{definition}\label{defn ICM}
     Let $\cY/\cO$ a normal quasiprojective flat $\cO$-scheme.  We say $\cY/\cO$ is an integral canonical model if:
     \begin{enumerate}
         \item $\cY$ admits a uniform resolution $\pi:\cS\to\cY$.
         \item For every closed point $v$ of $\Spec\cO$, the base-change $\cY_v/\cO_v$ satisfies the extension property over $\cO_v$.
     \end{enumerate}
 \end{definition}

Note that if $\cY/\cO$ is smooth, the condition on the existence of a uniform resolution is automatic, so this definition generalizes \Cref{defn smooth ICM}.  The next proposition generalizes \Cref{smooth extend=>general} to the normal case.

\begin{proposition}\label{extend maps}
    Let $\cO$ be a Dedekind domain with fraction field $F$.  Let $\cY/\cO$ be a normal quasiprojective flat $\cO$-scheme admitting a uniform resolution, $\cY'/\cO$ an integral canonical model, and $f_{F}:Y\to Y'$ a morphism over the generic point.
    \begin{enumerate}
        \item $f_F$ extends to a morphism $f:\cY\to\cY'$ if and only if there is a commutative diagram
            \[\begin{tikzcd}
            \piet(Y,\bar y)\ar[d]\ar[r]&\piet(Y',\bar y')\ar[d]\\
            \piet(\cY,\bar y)\ar[r,dashed]&\piet(\cY',\bar y')
        \end{tikzcd}\]
    for compatibly chosen basepoints.
    \item Suppose further that $\cY'$ admits a uniform resolution $\pi':\cS'\to\cY'$, a log smooth compactification of $\cS'$, and that $\cY'$ has a geometrically nonextendable $\ell$-adic local system $\V_\ell$.  Then $f_F$ extends to a morphism $f:\cY\to\cY'$ if and only if $f_F^*(\V_\ell)$ extends to $\cY$.
    \end{enumerate}

\end{proposition}
For the proof we need the following:

 \begin{lemma}[{Rigidity Lemma, \cite[Lemma 1.15]{debarre}}]\label{lem:rigiditylemma}  Let $X,Y,Z$ be integral finite type schemes, $f:X\to Y$ a proper morphism with $\cO_Y\xrightarrow{\cong}f_*\cO_X$ via the natural map, and $g:X\to Z$ a morphism which is constant on $f^{-1}(y)$ for some point $y$ of $Y$.  Then for an open neighborhood $y\in U\subset Y$, $g|_{f^{-1}(U)}:f^{-1}(U)\to Z$ factors through $f|_{f^{-1}(U)}:f^{-1}(U)\to Y$. 
    
\end{lemma}
\begin{proof}
    The proof is exactly the same as in \cite{debarre}.  Consider the (proper) morphism $h=f\times g:X\to Y\times Z$, and let $W$ be the image with projection map $p:W\to Y$.  Since $\cO_Y\subset p_*\cO_W\subset f_*\cO_X=\cO_Y$, the Stein factorization of $p$ is trivial.  Since $h^{-1}((p^{-1}(y)))=f^{-1}(y)$ maps to a point in $Z$, $p^{-1}(y)$ is a point, hence the fibers of $p$ are 0-dimensional in a neighborhood of $y$, and therefore $p:W\to Y$ is an isomorphism in a neighborhood of $y$. 
\end{proof}
\begin{corollary}\label{factors}
    Let $\cY,\cY'/\cO$ be normal finite type flat $\cO$-schemes, and $\pi:\cS\to\cY$ a uniform resolution.  Assume $\cY'/\cO$ is quasiprojective. Then any morphism $f:\cS\to\cY'$ which factors through $\pi_F:S\to Y$ over the generic point factors through $\pi$.
\end{corollary}

\begin{proof}Let $L'$ be a very ample bundle on $\cY'$, let $\cY^j\subset \cY$ be an $\cO$-flat stratum in a flattening stratification of $\cS\to\cY$, let $\cS^j\to\cY^j$ be the base-change to $\cY^j$, and consider the composition $\cS^j\to\cS\to\cY'$.  For any fiber $F$ of $\cS^j\to\cY^j$, $L|_{F}$ is globally generated, hence has a section which is generically nonzero on every geometric component.  On the other hand, $L^\vee$ is trivial on every geometric fiber of $\cS^j\to\cY^j$ over the generic point, hence by semicontinuity $L^\vee|_{F}$ has at least one nonzero section.  The product is a section of $\cO_{F}$ which is not identically zero, so since the fibers of $\cS^j\to\cY^j$ are geometrically connected it must be nowhere zero.  Thus, $L|_{F}$ is trivial, and so $f$ contracts every fiber of $\cS\to\cY$.  By the \Cref{lem:rigiditylemma}, we conclude that $f$ factors through $\cS\to\cY$.
\end{proof}
\begin{proof}[Proof of \Cref{extend maps}]

The forward implication of part 1 is trivial; for the reverse implication, by the smooth extension property we have an extension of the composition $S\to Y\to Y'$ to $\cS_v\to\cY'_v$ for every finite place, and therefore an extension $\cS\to\cY'$.  Then by \Cref{factors}, the map $\cS\to\cY'$ factors through the desired extension $f:\cY\to\cY'$.  Part 2 follows from part 1 by \Cref{faithful => natural}. 
    
\end{proof}
\begin{corollary}\label{uniqueness among models}
    Assume $Y$ admits a geometrically nonextendable local system $\V_\ell$.  Then:
    \begin{enumerate}
    \item There is at most one integral canonical model $\cY/\cO$ to which $\V_\ell$ extends and which admits a uniform resolution with a log smooth compactification.  
    \item Assume such an integral canonical model $\cY$ exists.  Then for any normal finite type flat model $\cZ/\cO$ to which $\V_\ell$ extends and which admits a uniform resolution, the identity map over the generic fiber extends to a morphism $\cZ\to\cY$.
    \end{enumerate}
\end{corollary}

\subsection{Descending and extending arithmetic local systems}Let $\cP/\cO$ be smooth with geometrically irreducible generic fiber, and $\cY/\cO$ normal finite type flat with geometrically irreducible generic fiber.  Let $f:\cP\to\cY$ be a proper morphism of $\cO$-schemes with $\cO_\cY\xrightarrow{\cong}f_*\cO_\cP$.  In this section we describe how arithmetic local systems can be extended to integral models and descended along $f$.

We start with the following:
\begin{lemma}\label{pietlem}
    Let $f:X\to Y$ be a proper morphism of schemes with $Y$ integral and assume $f$ has geometrically connected fibers and geometrically reduced generic fiber.  Then over a dense open $U\subset Y$ we have the following:
    \begin{enumerate}
        \item The base-change $X_U\to U$ induces an exact sequence
        \[\piet(X_{\bar \eta},\bar x)\to\piet(X_U,\bar x)\to\piet(U,\bar u)\to1\] for compatibly chosen basepoints, where $\eta$ is the generic point of $Y$.
        \item For every geometric point $\bar u$ of $U$, the specialization map $sp:\piet(X_{\bar\eta},\bar x)\to\piet(X_{\bar u},\bar x)$
        is surjective.
    \end{enumerate}
\end{lemma}
\begin{proof}
    According to \cite[\href{https://stacks.math.columbia.edu/tag/0579}{Tag 0579}]{stacks-project}, since the generic fiber is geometrically reduced we may take $U$ to be the interior of the set of points with geometrically reduced fibers.  We may further assume $X_U\to U$ is flat by generic flatness.  Then the first claim follows from the homotopy exact sequence \cite[\href{https://stacks.math.columbia.edu/tag/0C0J}{Tag 0C0J}]{stacks-project}
    and the second from base-change to the henselization of the local ring.
\end{proof}
The following is standard but we include it for completeness.  For a morphism of schemes $X\to Y$ which induces a surjection $\piet(X,\bar x)\to\piet(Y,\bar y)$ for compatibly chosen basepoints, we denote by $I_{X/Y}\subset\piet(X,\bar x)$ the kernel (somewhat abusively, since we suppress the basepoints from the notation).
\begin{lemma}The map $f_F:P\to Y$ induces the following commutative diagram with exact rows and columns for compatibly chosen basepoints.
\[\begin{tikzcd}
&1\ar[r]\ar[d]&1\ar[d]&&\\
&I_{P_{\bar F}/Y_{\bar F}}\ar[r,equals]\ar[d]&I_{P/Y}\ar[d]&&\\
    1\ar[r]&\piet( P_{\bar F},\bar p)\ar[d]\ar[r]&\piet(P,\bar p)\ar[r]\ar[d]&\piet(F,\bar F)\ar[r]\ar[d,equals]&1\\
    1\ar[r]&\piet( Y_{\bar F},\bar y)\ar[d]\ar[r]&\piet(Y,\bar y)\ar[d]\ar[r]&\piet(F,\bar F)\ar[r]&1\\
    &1&1&&
\end{tikzcd}\]
    
\end{lemma}
\begin{proof}
    The exactness of the rows is standard and the surjectivity of the vertical maps to the bottom row is a consequence of normality and the generic fiber being geometrically connected.
\end{proof}
\begin{corollary}\label{descend to field}
    Let $\V_\ell$ be an $\Z_\ell$-local system on $P$ for which the geometric local system $\bar \V_\ell$ on $ P_{\bar F}$ descends to $ Y_{\bar F}$.  Then $\V_\ell$ descends to $Y$.
\end{corollary}
\begin{lemma}
    After inverting finitely many primes of $\cO$, any $\ell$-adic local system $\V_\ell$ on $\cP$ which geometrically descends to $Y_{\bar F}$ descends to $\cY$.  Moreover, the same is true for any further localization of $\cO$.
\end{lemma}
\begin{proof} We first claim that after inverting finitely many primes, the following is true:  for every geometric point $\bar y_0$ of $\cY$ there is a geometric point $\bar y$ of $Y$ specializing to it such that the specialization map $\piet(\cP_{\bar y},\bar y)\to\piet(\cP_{\bar y_0},\bar y_0)$ is surjective.  Using \Cref{pietlem}, there is a stratification $\cY^i$ of $\cY$ such that this is true for $(\cP_{\cY^i})^{\mathrm{red}}\to\cY^i$, provided $\cY^i$ is flat over $\cO$.  After shrinking $\cO$, we may assume this is the case for all $\cY^i$, and therefore the claim is proven.

By \Cref{descend to field}, we may assume $\V_\ell$ descends to $Y$.  To finish it suffices to show any finite \'etale cover $\cP'\to\cP$ which descends to $Y$ over $P$ descends to $\cY$.  The descent clearly patches, so we may assume $\cY=\Spec R$ for a local ring $R$.  Base-changing to the henselization $\bar R$, the cover $\cP'_R\to\cP_R$ descends to $R$ by \cite[\href{https://stacks.math.columbia.edu/tag/0A48}{Tag 0A48}]{stacks-project}, since it is trivial on every geometric fiber of $\cP_R\to\Spec R$ by the above surjection of specialization maps.  The descent to $R$ comes equipped with descent data which gives a finite \'etale cover of $\Spec R$ which pulls back to $\cP'$.

\end{proof}
\begin{corollary}\label{descend to integral}
        After inverting finitely primes of $\cO$, the following is true.  For any $\ell$ and any $n$, any $\Z_\ell$-local system $\V_\ell$ on $\cP[1/n]$ which geometrically descends to $Y_{\bar F}$ descends to $\cY[1/n]$.  
\end{corollary}

\subsection{Relation to Hodge theory over $\bbC$}\label{hodge sect}
We end this section by summarizing some classical extension results in Hodge theory using the above language.  This is both motivation for the using extension properties to determine integral models and is needed to prove the characteristic 0 part of the extension in \Cref{thm ICM hodge}.
\begin{lemma}\label{griffiths extension}
    Let $P$ be a smooth complex algebraic variety equipped with a polarizable integral variation of Hodge structures $(W_\Z,F^\bullet W)$, and let $\Phi:P^\an\to[\mathbf{G}(\Z)\backslash D]$ be the period map.
    \begin{enumerate}
        \item For any partial log smooth compactification $(P',D)$, $(W_\Z,F^\bullet)$ extends to $P'$ if and only if the monodromy of $W_\Z$ around each component of $D$ is trivial.
        \item $\phi$ is proper if and only if and only if $V_\Z$ is nonextendable\footnote{We use the same definition for a $\Z$-local system}.
        \item There is a finite \'etale cover $\pi:\tilde P\to P$ and a partial log smooth compactification $\tilde P'$ of $\tilde P$ such that $\pi^*W_\Z$ extends to $\tilde P'$ where it is nonextendable.
        \item If the local monodromy of $W_\Z$ is unipotent, then there is a partial log smooth compactification $P'$ of $P$ such that $W_\Z$ extends to $P'$ where it is nonextendable.
    \end{enumerate}
\end{lemma}
\begin{proof}
    (1) is a result of Griffiths \cite[Theorem 9.5]{griffithsiii}.  For (2), (3), and (4), see \cite[\S3]{brunebarbe}.
\end{proof}
For simplicity, we will assume $W_\Z$ has torsion-free monodromy.  In this case, it follows from the lemma that for any period map $\Phi: P^\an\to [\mathbf{G}(\Z)\backslash D]$, there is a partial log smooth compactification $P'$ of $P$, a proper extension $\Phi':P'^\an\to [\mathbf{G}(\Z)\backslash D]$ which factors through a finite inertia-free cover $\Gamma\backslash D$, and a proper algebraic map $h':P'\to Y$ to a normal variety $Y$ with $\cO_Y\xrightarrow{\cong } h'_*\cO_{P'}$ such that $\Phi'$ factors as $P'^\an\xrightarrow{h'^\an} Y^\an\xrightarrow{\psi^\an}[\mathbf{G}\backslash D]$.  We call $Y$ the Stein factorization of the period map $\Phi$.  Observe that the local system $W_\Z$ as well as the associated filtered flat vector bundle $(W,\nabla,F^\bullet W)$ are pulled back from $V_\Z$ and $(V,\nabla,F^\bullet V)$ on $Y$.
\begin{corollary}\label{char 0 extension}
    Let $Y$ be the Stein factorization of a period map $(W_\Z,F^\bullet W)$ on $P$ with torsion-free monodromy.  Let $X$ be a smooth complex algebraic variety and $U\subset X$ a dense Zariski open set.  Then a morphism $g_U:U\to Y$ extends to $g:X\to Y$ if and only if $g_U^*\W_\Z$ extends to $X$.
\end{corollary}

\begin{definition}\label{defn:griffiths}
    For any scheme $Y$ and a vector bundle $(V,F^\bullet V)$ on $Y$ with a locally split decreasing filtration, we define the \textit{Griffiths bundle} of $(V,F^{\bullet}V)$ to be
    $\otimes_{i\in\Z}\det F^i V$.
\end{definition}

The main result of \cite{bbt} is the following:
\begin{theorem}[\cite{bbt}]\label{thm:bbt}  If $Y$ is the Stein factorization of a period map as above, then the Griffiths bundle on $Y$ is ample.
    
\end{theorem}

\begin{lemma}
    Suppose $P$ is defined over a number field $F$.
    \begin{enumerate}
        \item If the filtered flat vector bundle $(W,\nabla,F^\bullet W)$ underling $(W_\Z,F^\bullet W)$ is defined over $F$, so is the Stein factorization of the (compactified) period map $h':P'\to Y$.  Moreover, the descent $(V,\nabla,F^\bullet V)$ is also defined over $F$.
        \item If $W_{\Z_\ell}$ extends to an arithmetic local system $_{et}\W_\ell$ on $P$, then $V_{\Z_\ell}$ extends to an arithmetic local system $_{et}\V_\ell$ on $Y$.
    \end{enumerate}
\end{lemma}


\section{Extension in codimension 2}\label{sect:codim 2}

In this section we prove our main extension result in codimension 2.

\begin{theorem}\label{thm:extncodim2}Let $F/\Q_p$ be a finite unramified extension with ring of integers $\cO$.  Let $\cS/\cO$ be a smooth $\cO$-scheme equipped with a Fontaine-Laffaille module $\bfM\in\mathrm{MF}_{[0,w]}(\hat \cS)$ on the $p$-adic completion $\hat\cS$ with underlying filtered flat bundle $(M,\nabla,F^\bullet M)$ where $0\leq w\leq p-2$.  Let $\cX/\cO$ be a smooth finite type $\cO$-scheme with $\dim\cX=2$, $x\in \cX$ a point in the special fiber, and $f:\cX\setminus x\to \cS$ a morphism.   Suppose we have a commutative diagram

    \[
   \begin{tikzcd}
       &\cX'\ar[d,"\pi"]\ar[r,"g"]&\cS\\
       \cX\setminus x\ar[ur]\ar[r]&\cX&
   \end{tikzcd}
   \]
 where
 $\pi:\cX'\to\cX$ is an iterated blow-up at closed points lying above $x$, and $g:\cX'\to \cS$ is a morphism extending $f$ via the natural lift $\cX\setminus x\to\cX'$.  Then the Griffiths bundle of $\hat g^*(M,F^\bullet M)$ is trivial on $\pi^{-1}(x)$. 

 In particular, if the Griffiths bundle of $(M,F^\bullet M)$ is ample on $\hat \cS$, then the morphism $f$ extends over $x$. 
\end{theorem}
We recall the notion of Fontaine--Laffaille modules and the category $\mathrm{MF}_{[0,w]}(\hat \cS)$ in the next section.

\subsection{Fontaine-Laffaille modules}

\subsubsection{Reminder on crystals and flat bundles}
Let $F/\Q_p$ be an unramified extension and let $\cO$ be the ring of integers of $F$.  We remind the reader that for a smooth formal scheme $\cX/\Spf\cO$ there is a natural equivalence of categories between crystals in vector bundles on the special fiber $\cX_0$ and flat vector bundles on $\cX$ (see for example \cite[\href{https://stacks.math.columbia.edu/tag/07JH}{Tag 07JH}]{stacks-project}).  In particular, there is a well-defined ``Frobenius pull-back'' of a flat vector bundle $(V,\nabla)$ on $\cX$ which we denote $F^*(V,\nabla)$.  Concretely, for a lift of Frobenius $\phi:\cX\to\cX$, we form $(\phi^*V,\phi^*\nabla)$, and for any other lift of Frobenius $\phi':\cX\to\cX$ there is a canonical isomorphism of flat vector bundles $(\phi'^*V,\phi'^*\nabla )\to(\phi^*V,\phi^*\nabla)$ given by the formula:
\begin{equation}\label{taylor1}1\otimes v\mapsto \sum_I \frac{(\phi'(t)-\phi(t))^I}{I!}\otimes \nabla(\partial)^Iv.\end{equation}
Here we take $t_1,\dots,t_n$ to be local co-ordinates for $\cX$ with a dual basis of derivations $\partial_i$ and for a multi-index $I=(i_1,\dots,i_n)$ we set $\nabla(\partial)^I=\prod_j \nabla(\partial_i)^{i_j}$.

The above isomorphisms satisfy an appropriate cocycle condition, hence local pullbacks (where a lift of Frobenius exists) canonically glue.  Moreover, for any morphism $f:\cY\to\cX$ of smooth formal schemes there is a natural identification $f^*F^*(V,\nabla)\cong F^*f^*(V,\nabla)$ by factoring $f$ as the graph followed by the projection $\cX\to\cX\times\cY\to\cY$ and using the existence of campatible lifts of Frobenius for the graph and the projection separately.

\subsubsection{Unramified case}
Let $F/\Q_p$ be an unramified extension and let $\cO$ be the ring of integers of $F$. Let $\cX$ denote a smooth formal scheme over $\cO$, and $\cX^{rig}$ its rigid-analytic fiber.  In \cite{faltings} Faltings defines a category of Fontaine-Laffaille modules as follows:

\begin{definition}\label{def:MFunramified}
Assume first that $\cX$ admits a lift of Frobenius $\phi:\cX\ra \cX$. Let $0\leq w\leq p-2$. Then an object of $\MF_{[0,w]}(\cX)$ consists of a vector bundle $M$ equipped with a decreasing filtration $F^{\bullet}M$ with $F^0M=M, F^{a+1}M=0$, a Griffiths-transverse flat connection $\nabla_M$, and a morphism $\phi_M:\phi^*M \ra M$ satisfying:

\begin{enumerate}
    \item $\phi_M$ is horizontal for $\nabla_M$
    \item The restriction of $\phi_M$ to $\phi_M^*F^qM$ is divisible by $p^q$, for $q\in[0,w]$. 
    \item $M=\sum_{i=0}^w p^{-i}\phi_M(\phi_M^*F^i M)$.
\end{enumerate}
Morphisms in $\MF_{[0,w]}(\cX)$ are morphisms of the underlying modules which are compatible the additional structures.

Faltings shows that this category is independent of the choice of $\phi$ up to a natural isomorphism, using the canonical identifications of \cref{taylor1}. Indeed, for a different choice of Frobenius lift $\phi':\cX\ra \cX$, we obtain a new lift $\phi_M'$ to $M$ by setting: 
\begin{equation}\label{eq:frobeniuschange}
\phi'_M(1\otimes m):=\sum_{I} \phi(1\otimes \nabla_M(\partial^I)(m))\frac{(\phi'(t)-\phi(t))^I}{I!}.
\end{equation}

Finally, for a general $\cX$, one simply covers by affines where a lift of Frobenius exists as above and glues in the natural way. Since Frobenius lifts lift uniquely under etale maps, it follows that $\MF_{[0,w]}(\cX)$ satisfies \'etale descent.
    
\end{definition}

Whenever we write $\MF_{[0,w]}$ we assume that $0\leq w\leq p-2$. The following lemma must be well-known to the experts but we cannot find it in the literature in this form:

\begin{lemma}\label{lem:unrbasechangeFL}

Let $g:\cX\ra \cY$ be a map of smooth formal $\cO$-schemes. There is a natural base-change morphism $\MF(g):\MF_{[0,w]}(\cY)\ra \MF_{[0,w]}(\cX)$, compatible with compositions, so that $\MF(g_1\circ g_2)=\MF(g_1)\circ\MF(g_2).$

\end{lemma}

\begin{proof}
    The statement is clearly Zariski-local on both $\cX$ and $\cY$, which we may therefore assume to be affine. We may simply pull back the vector bundle, filtration, and connection. The only tricky part lies in the semi-linear Frobenius morphism $\phi$.
    As explained in \cite{faltings}, the data of $(M,\nabla_M,\phi_M)$ is equivalent to an $F$-crystal on $\cX_0$. Since $F$-crystals pull back naturally, the claim is proven.

\end{proof}

\subsubsection{Ramified case}

Let $F'/\Q_p$ be a finite extension, which might be ramified. Let $F\subset F'$ be the maximal unramified subfield, and let $\cO',\cO$ be the rings of integers of $F',F$. Let $\pi\in \cO$ denote a uniformizer, with minimal polynomial $f(t)\in \cO[t]$. Let $R=\cO[[t]][\frac{f^n}{n!}]$, and let $R_{\cO'}$ be the PD-completion of $R$. Then $R_{\cO'}$ naturally acquires a decreasing filtration with $F^0R_{\cO'}=R_{\cO'}$ and $F^qR_{\cO'}$ being generated by the divided powers $\frac{f^n}{n!}, n\geq q$. 

In \cite[Def 2]{faltingsramified} Faltings defines a category $\MF_{[0,w]}(\cO')$ for $a\leq p-2$ whose definition we recall here:

\begin{definition}\label{def:MFramified}
    An object $\MF_{[0,w]}(\cO')$ is a $R_{\cO'}$-module $M$ with a decreasing filtration $F^\bullet M$, a Griffiths-transverse flat connection $\nabla_M$ which nilpotent modulo $p$, and a (semi-linear) Frobenius $\phi_M$ satisfying the following:
    \begin{enumerate}
        \item $M$ is filtered-free as an $R_{\cO'}$-module (meaning a direct sum of twists of $R_{\cO'}$), with a basis $m_i$ in degrees $q_i$ in $[0,w]$.
        \item $\phi_M$ is horizontal for $\nabla_M$
        \item The restriction of $\phi$ to $F^qM$ is divisible by $p^q$, for $q\in[0,w]$. 
        \item The elements $\phi(m_i)/p^{q_i}$ form another basis for $M$.
    \end{enumerate}
    Morphisms in $\MF_{[0,w]}(\cO')$ are morphisms of the underlying modules which are compatible the additional structures.
\end{definition}

Once again, as Faltings observes, the data of $(M,\nabla_M,\phi)$ gives an $F$-crystal on $\cO'\otimes\F_p$.

\begin{lemma}\label{lem:rambasechangeFL}
Let $\cS$ be a smooth formal scheme over $\Spf \cO$, and let $\iota:\Spf \cO'\ra \cS$ be a morphism of formal $\cO$-schemes. There is a
natural base-change morphism $\MF(\iota):\MF_{[0,w]}(\cS)\ra \MF_{[0,w]}(\cO')$.   
\end{lemma}

\begin{proof}
Let $\bfM\in\MF_{[0,w]}(\cS)$. We may define the flat bundle $(M_1,\nabla_{M_1})$ underlying $\MF(\iota)(\bfM)$ and the Frobenius $\phi_{M_1}$ by simply pulling back the $F$-crystal underlying $M$ to $\Spec \cO'$. It remains to define the filtration. Without loss of generality, we may assume $\cS$ is affine, so $\iota:A\ra \cO'$ for $A$ the p-adic completion of a smooth $\cO$-algebra. By smoothness, we may pick a lift $\iota_1:A\ra R_{\cO'}$, and note $M_1\cong R_{\cO'}\otimes_A M$. We now get a natural filtration on $M_1$ by taking the product filtration.  We must show that this is independent of the choice of $\iota_1$, so let $\iota_2$ be a different lift. Then the crystal structure gives an isomorphism of flat bundles $M_1\ra M_2$ given as follows. As before, let $t_1,\dots,t_n$ be local co-ordinates for $X$ with a dual basis of derivations $\partial_i$.
For a multi-index $I=(i_1,\dots,i_n)$ set $\nabla_M(\partial)^I=\prod_j \nabla_M(\partial_i)^{i_j}$.

Then our isomorphism $M_1\ra M_2$ is given by:
$$1\otimes_{\iota_1} m\mapsto \sum_{I} \frac{(\iota_1(t)-\iota_2(t))^I}{I!}\otimes_{\iota_2}\nabla_M(\partial^I) (m) .$$

Since $\frac{(\iota_1(t)-\iota_2(t))^I}{I!} \in F^{|I|}R_{\cO'}$ and the connection is Griffiths transverse, it follows that this isomorphism respects filtrations. This completes the proof

\end{proof}

We point out that if $C_{\bfM}$ denotes the crystal corresponding to $\bfM\in \MF_{[0,w]}(\cO')$, then $C_{\bfM}(\Spf\cO')\cong \cO'\otimes_{R_{\cO'}} M$.


\def\Frac{\operatorname{Frac}}
\subsubsection{Local systems}
Following Faltings, we say that a \emph{small} $\cO$-algebra is one that admits a finite, etale map from $\cO\langle T_1^{\pm},\dots,T_n^{\pm}\rangle$. 
Let $\cS$ be a smooth formal $\cO$-scheme. Using the above notion of Fontaine--Laffaille module, in \cite[Thm 5]{faltingsramified} there is defined a fully faithful functor 
 $\mathbb{D}:\MF_{[0,w]}(\cS)\ra \Loc_{S^{rig}}(\Z_p)$. We recall the notion here, and to do so we recall some basic rings that are used:
 
 Let $\cS=\Spf A$ where $A$ is a small $\cO$-algebra. Consider the maximal extension of the fraction field $\Frac(A)$ in which $A[1/p]$ has unramified normalization, and let $\ol{A}$ be the normalization of $A$ in this field. We let $A^{\flat}$ be the perfection of $\ol{A}/p$ and $W(A^{\flat})$ the Witt ring. Then there is a map $W(A^{\flat})\ra \widehat{\ol{A}}$, and we let $B^+(A)$ denote the completed divided power hull with respect to the ideal generated by  the kernel $I$ of this map and $p$. The group $G_A:=\Gal(\ol{A}[1/p]/A[1/p])$ acts continuously on $B^+(A)$.

 Giving an $\bfM\in\MF_{[0,w]}(\cS)$ gives an $F$-crystal $C_\bfM$ on $\cS_0$, and hence by pull
 back an $F$-crystal $\ol{C}_\bfM$ on $\widehat{\ol{A}}\otimes\F_p$ with a $\Gal(\ol{A}/A)$  action. As in \cref{lem:rambasechangeFL}, this endows $\ol{C}_\bfM(B^+(A))$ with a natural filtration. We define
 $$\mathbb{D}(\bfM):=\Hom_{B^+(A),\Fil,\phi}(\ol{C}_\bfM(B^+(A)),B^+(A))$$ where the homomorphism is of $B^+(A)$-modules and is required to respect the filtration and Frobenius.

 The analogous definition using $B^+(\cO')$ gives a functor $\D:\MF_{[0,w]}(\cO')\ra \Loc_{\cO'}(\Z_p)$, where $\Loc_{\cO'}(\Z_p)$ simply denotes free $\Z_p$ representations of $\Gal_{F'}$. 
 \begin{lemma}\label{lem:FLpullback}
     Let $\cS$ be a smooth formal $\cO$-scheme, and $\cX$ be either $\Spf \cO'$ or a smooth formal $\cO$-scheme, and $f:X\ra \cS$. Let $\bfM\in\MF_{[0,w]}(\cS)$. Then we have a natural isomorphism 
     $f^{-1}\D(M)\ra \D(\MF(f)(M))$.
 \end{lemma}

 \begin{proof}
     Assume first that $\cX$ is a smooth formal $\cO$-scheme, and assume $\cX=\Spf R,\cS=\Spf A$ are both small and affine. Now we have a natural map $B^+(A)\ra B^+(R)$ compatible with the Galois actions.  Moreover, since $M$ is a vector bundle, we have that 
     $\ol{C}_\bfM(B^+(A))\otimes B^+(R)\cong \ol{C}_\bfM(B^+(R))$.  We therefore obtain a natural morphism $f^{-1}\D(\bfM)\ra \D(\MF(f)(\bfM))$. It remains to show that this is an isomorphism.

     By \cite[II,h]{faltings} we see that
     the natural map $\ol{C}_\bfM(B^+(A))\ra \mathbb{D}(\bfM)^*\otimes B^+(A)$ is an isomorphism up to $\beta_0^{a}$ where $\beta_0$ is as in \cite{faltings}. It follows that 
     $$\beta_0^{a}\cdot \D(\MF(f)(\bfM))\otimes B^+(R)\subset f^{-1}\D(\bfM)\otimes B^+(R)\subset \D(\MF(f)(\bfM))\otimes B^+(R).$$

     Finally, since $p\nmid \beta_0^a$, it follows that the map $f^{-1}\D(\bfM)\ra \D(\MF(f)(\bfM))$ is an isomorphism as desired.
     Since this map is clearly natural, we may glue up to obtain the claim for a general smooth formal $\cO_0$ schemes $X,S$.

     If $X=\Spf\cO'$ the proof is identical, using $B^+(\cO')$ instead of $B^+(R)$ and \cite[Thm 5]{faltingsramified} instead of \cite[II,h]{faltings},

 \end{proof}

\def\scrX{\mathscr{X}}
\subsubsection{Liu--Zhu correspondence}

Let $\scrX$ denote a smooth geometrically connected rigid analytic variety over a finite extension $F'/\Q_p$. Let $\bL$ denote a de Rham $\Q_p$-local system on the etale site $\scrX_{\et}$. Then \cite[Thm 3.9]{liuzhu} associate a filtered flat bundle with Griffiths transverse connection $LZ(\bL)$ in a functorial way with respect to both $\scrX$ and $\bL$. If $\scrX=\cX^{rig}$ for a smooth, small formal scheme $\cX=\Spf R$, then $LZ(\bL)\cong (\bL\otimes \cO B_{\dR}(R[1/p]))^{G_R}$. 

\begin{lemma}\label{lem:FLtoLZ}
Let $\cX$ denote either a smooth formal scheme over $\cO$ or $\Spf\cO'$. Let $\bfM\in \MF_{[0,w]}(\cX)$. Then there is a canonical isomorphism
\[\begin{cases}
    M[1/p]\ra LZ(\D(M)^*) & \textrm{ if } \cX \textrm{ is smooth over }  \cO\\
    M\otimes_{R_{\cO'}} F'\ra LZ(\D(M)^*) & \textrm{ if } \cX =\Spf\cO'\\
\end{cases}\]
\end{lemma}

\begin{proof}

We once again first work locally on  $\Spf R$ for a small $R$. By \cite[Thm 3.9,v]{liuzhu}, we see that 
$$LZ(\D(\bfM)^*)\otimes_R \cO B_{\dR}(R[1/p])\ra \D(\bfM)^*\otimes \cO B_{\dR}(R[1/p]))$$ is an isomorphism

Moreover, there is a natural map \cite{Brinon} $B^+(R)[\beta_0^{-1}]\ra \cO B_{\dR}(R[1/p])$. Hence, composing the above with \cite[II, Theorem 2.6*, h]{faltings} we see that $$LZ(\D(\bfM)^*)\otimes \cO B_{\dR}(R[1/p]) \cong C_{\bfM}(B^+(R))\otimes_{B^+(R)} \cO B_{\dR}(R[1/p]). $$


Now $\cO B_{\dR}(R[1/p])$ maps to $\widehat{\ol{R}}[1/p]$, we obtain a Galois-equivariant morphism 

$$LZ(\D(\bfM)^*) \otimes \hat{\overline{R}}[1/p] \cong \ol{C}_{\bfM}(B^+(R))\otimes_{B^+(R)} \hat{\overline{R}}[1/p] \cong \ol{C}_{\bfM}(\widehat{\ol{R}})[1/p] \cong M\otimes_R \widehat{\ol{R}}[1/p].$$
Taking Galois invariants, we see that $LZ(\D(\bfM)^*) \cong M[1/p]$. Since all the maps are natural, this glues up to completes the proof for an arbitrary smooth formal scheme $\cO$.

$$LZ(\D(\bfM)^*)\cong \left(\ol{C}_{\bfM}(B^+(R))\otimes_{B^+(R)} \cO B_{\dR}(R[1/p])\right)^{G_R}\ra \ol{C}_{\bfM}(\widehat{\ol{R}})[1/p]^{\Gal(R)}\cong  (M\otimes_R \widehat{\ol{R}}[1/p])^{\Gal(R)}\cong M[1/p]$$

For $X=\cO'$ the proof is exactly the same, noting that in this case for $\bfM\in\MF_{[0,w]}(\cO')$, we have $C_{\bfM}(R)\cong M\otimes_{R_{\cO'}} \cO'$. 
\end{proof}

\def\scrY{\mathscr{Y}}
\def\spa{\operatorname{spa}}
\def\bbL{\mathbb{L}}
\def\bbM{\mathbb{M}}
\def\bbN{\mathbb{N}}
\subsection{Proof of Theorem \ref{thm:extncodim2}}

We use the notion of crystallinity in \cite[Def 4.4]{guo}, which we refer to as \emph{Brinon-crystalline}. Recall that a \emph{classical point} of a rigid space $\scrY$ over $F$ is a map $\spa F'\ra \scrY$ for a finite extension $F'/F$. 

\begin{lemma}\label{lem:pointwisecrystaline}
    Let $\cO$ be the ring of integers in a finite unramified extension $F/\Q_p$, and let $\cX/\Spf\cO$ be a smooth formal scheme. Let $\bbL$ be a $\Z_p$-local system on $\cX^{rig}$ which is Brinon-crystalline at every classical point of $\cX^{rig}$. Then $\bbL$ is crystalline.
\end{lemma}

\begin{proof}
      By \cite[Thm 1.1]{guo} we obtain that $\bbL$ is Brinon-crystalline. Since $F/\Q_p$ is unramified, it follows by \cite[Thm 3]{Hokaj} that $\bbL$ is crystalline.
\end{proof}

\begin{proof}[Proof of \Cref{thm:extncodim2}]
    Let $\bbM$ be the $\Z_p$-local system on $\cS^{rig}$ corresponding to $\bfM$ on $\hat\cS$, and $\bbN=(f^{rig})^*\bbM$ the pullback.  By \Cref{lem:pointwisecrystaline}, $\bbN$ is crystalline, hence corresponds to a Fontaine--Laffaille module $\mathbf{N}$ on $\hat \cX$.
    Now by Lemma \ref{lem:FLtoLZ} we see that we have maps $r_1:g^*M\ra LZ(g^{-1}\D(\bfM)^*)$ and $r_2:\pi^*N\ra LZ(\pi^{-1}\D(\bfN)^*)$, and natural isomorphisms $g^{-1}D(\bfM)^*\ra \pi^{-1}D(\bfN)^*$ giving induced isomorphisms  $$s:LZ(g^{-1}\D(\bfM)^*)\ra LZ(\pi^{-1}\D(\bfN)^*).$$ We claim that the images of the maps $r_1,r_2$ are identified by $s$.  To prove this its enough to work at every $\ol{\Q_p}$-point. So let $F'/\Q_p$ be a local field and $\cO'$ its ring of integers and $\iota:\Spf \cO'\ra \cX'$ an arbitrary point. Note that by \Cref{lem:FLpullback} we have that $\MF(g\circ\iota)(\bfM)$ is a Fontaine-Laffaille module, and $\D(\MF(g\circ\iota)(\bfM))\cong \iota^{-1}g^{-1}\D(\bfM)$, and that $\D(\MF(g\circ\iota)(\bfN))\cong \iota^{-1}\pi^{-1}\D(\bfN)$.

    By Lemmas \ref{lem:FLpullback}, \ref{lem:FLtoLZ} we obtain maps $$\iota^*r_1:\iota^*g^*M\ra \iota^*LZ(g^{-1}\D(\bfM)^*)\cong LZ(\D(\MF(g\circ\iota)(\bfM))^*).$$
    Moreover, this map is canonical and therefore so is its image. It follows that the images of $\iota^*r_1$ and $\iota^*r_2$ agree. Thus, the images of $r_1$ and $r_2$ agree, as desired. It follows that $g^*M$ and $\pi^*M$ are isomorphic as filtered flat bundles, and thus the Griffiths bundle of $g^*(M,\nabla, F^\bullet M)$ is trivial on $\pi^{-1}(x)$, as desired. 
    
    It follows by \Cref{lem:rigiditylemma} that if the Griffiths bundle of $(M,\nabla, F^\bullet M)$ is ample, the map $g$ descends to $\cX$, completing the proof.

\end{proof}

\section{Criteria for local integral canonical models}\label{localcanonicalmodels}
\def\bG{\mathbf{G}}
\def\bZ{\mathbb{Z}}
The goal of this section is to give a general setting in which local integral canonical models exists, in the sense of \Cref{defn ICM}.  For simplicity, we first give the smooth case.

\begin{theorem}\label{local ICM smooth}
    Let $F/\Q_p$ be a finite unramified extension with ring of integers $\cO$.  Let $\cS/\cO$ be a smooth $\cO$-scheme equipped with:
    \begin{enumerate}
        \item A log smooth compactification $\bar\cS/\cO$ of $\cS$ with natural boundary.
        \item A Fontaine-Laffaille module $\bfM\in\mathrm{MF}_{[0,w]}(\hat\cS)$ with ample Griffiths bundle.
    \end{enumerate}
    Then if $p>w+1$, $\cS$ is an integral canonical model.
\end{theorem}
See \Cref{defn nat bound} and \Cref{defn:griffiths} for the definitions of ($\pi$-)natural boundary and the Griffiths bundle.  In the normal case, we will additionally need a form of stratified resolution.

\begin{definition}\label{strat res}
    Let $\cO$ be a Dedekind domain and $\cY/\cO$ a normal finite type flat $\cO$-scheme.  A stratified resolution of $\cY$ consists of:
    \begin{enumerate}
        \item An increasing sequence of closed subspaces $\cY^1\subset \ldots\subset\cY^m=\cY$, each flat over $\cO$, such that $\cY^j\setminus\cY^{j-1}$ is smooth over $\cO$ for each $j$.  Let $\cY^{\prime j}\to\cY^j$ be the normalization.
            \item For each $j$, a resolution $\pi'^j:\cS^j\to\cY^{\prime j}$ and a log smooth compactification $\bar \cS^j/\cO$ with the property that the composition $\cS^j\to\cY^{\prime j}\to\cY^j$ (which we call $\pi^j:\cS^j\to \cY^j$) is an isomorphism over $\cY^j\setminus\cY^{j-1}$.
            \end{enumerate}
            We say the stratified resolution:
            \begin{enumerate}\setcounter{enumi}{2}
                \item is \emph{uniform} if the resolution of the top stratum $\pi^m:\cS^m\to\cY^m=\cY$ is uniform.
                \item \emph{has natural boundary} if each $\bar\cS^j$ has $\pi^j$-natural boundary.
            \end{enumerate}

\end{definition}

\begin{theorem}\label{local ICM normal}
Let $F/\Q_p$ be a finite unramified extension with ring of integers $\cO$.  Let $\cY/\cO$ be a normal finite type flat $\cO$-scheme equipped with:
\begin{enumerate}
    \item A uniform stratified resolution $\{\pi^j:\cS^j\to\cY^j\}$ with natural boundary.
    \item A filtered bundle $(V,F^\bullet V)$ on $\hat \cY$ with ample Griffiths bundle.
    \item For each $j$ a Fontaine-Laffaille module $\bfM^j\in\mathrm{MF}_{[0,w^j]}(\hat \cS^j)$ and a proper dominant generically finite morphism $q^j:\cT^j\to\cS^j$ from a smooth $\cT^j/\cO$ such that 
    \begin{align*}
       \hat q^{j*}(M^j,F^\bullet M^j)&\cong \hat q^{j*}\hat \pi^{j*}(V,F^\bullet V)
    \end{align*}
     as filtered vector bundles.
\end{enumerate}
    Then if $p>\max(w^j)+1$, $\cY$ is an integral canonical model.
\end{theorem}
Note that \Cref{local ICM normal} directly implies \Cref{local ICM smooth} by taking the trivial stratification and $\cT^m=\cS^m=\cY=\cS$.

\def\bfN{\mathbf{N}}
\subsection{Extension in codimension $\geq 2$}
We now use the results of \Cref{sect:codim 2} to show that, in the presence of a Fontaine--Laffaille module with ample Griffiths bundle, morphisms from smooth sources which are defined outside of a set of codimension $\geq 2$ extend.

 \begin{proposition}\label{prop: codim2}
     Assume the setup of \Cref{local ICM normal}.  Let $\cX/\cO$ be a smooth finite type $\cO$-scheme and $f:\cX\setminus\cZ\to \cY$ a morphism for $\cZ\subset  \cX$ a subset of codimension $\geq2$ which is contained in the special fiber.  Then there is an extension $f:\cX\to\cY$. 
 \end{proposition}
 \begin{proof}
 If $j$ is minimal such that $Y^j$ contains the image of the generic fiber $X$, then the image of $X$ meets $Y^j\setminus Y^{j-1}$, so there is a dense open set of $X$ which lifts to $S^j$.  By purity of the branch locus, $\cX\setminus\cZ\to\cX$ induces an isomorphism on fundamental groups, so the map $X\to Y$ is $\cO$-admissible.  It follows by \Cref{extend codim 1} that there is a finite type flat $\cX'/\cO$, a proper morphism $\pi:\cX'\to \cX$ which is an isomorphism on a dense open subset, and a morphism $\cX'\to \cS^j$ whose generic fiber lifts the morphism $X\to Y^j$.
 
There is nothing to prove for $\dim \cX\leq 1$.  For $\dim \cX\geq 2$, we argue as in \cite{dejongoort} (see also \cite{moretbailly,faltingschai}) by induction on $\dim \cX$ starting with the base case $\dim \cX=2$ and using deformation theory for the induction step.    

Assume first that $\dim\cX=2$, so there is a finite set of points $\Sigma $ in the special fiber of $\cX$ such that $\pi$ is an isomorphism on the complement $\cX\setminus\Sigma$.  By resolution of singularities for surfaces (\cite{lipman}, specifically \cite[\href{https://stacks.math.columbia.edu/tag/0AHI}{Tag 0AHI}]{stacks-project}), there is a composition of blow-ups at points lying above $\Sigma$, $\cX_n\to \cdots\to \cX_1\to \cX$, whose composition factors as $\cX_n\to \cX'\to \cX$.  Note that the map $\cX\setminus \Sigma\to \cY$ lifts to $\cS^j$.  To show that $\cX'\to\cS^j\to\cY$ factors through $\cX'\to\cX$ (and hence the existence of the required extension), it suffices (for instance by \Cref{lem:rigiditylemma}) to show the following:
   \vskip.6em
   \begin{claim}
       The composition $\cX_n\to\cX'\to\cS^j\to\cY$ contracts all exceptional curves.
   \end{claim}

   \begin{proof}  Applying \Cref{thm:extncodim2} to the diagram
   \[
   \begin{tikzcd}
       &\cX_n\ar[d]\ar[r]&\cS^j\\
        \cX\setminus\Sigma\ar[ur]\ar[r]&\cX&
   \end{tikzcd}
   \]
   we obtain that the Griffiths bundle of $\bfM^j$ is trivial on every exceptional divisor $E$.  For each $E$ we may take a curve $C$ in $\cT^j$ dominating $E$, and it follows that the Griffiths bundle of $V$ is trivial on $C$, hence it is degree 0 on $E$.  Thus, every exceptional divisor is contracted in $\cY$.
   \end{proof}

 We now treat the general case $\dim \cX>2$ by induction via deformation theory.  It will be sufficient to show that for any $x\in \cZ$, $f$ extends on an \'etale neighborhood of $x$.  In particular we may assume $\cX$ is affine and $\cY$ is a closed subscheme of a smooth affine $\cO$-scheme $\cV$.    Let $\cH\subset \cX$ be a divisor containing $x$ such that $\cH\cap\cZ$ is codimension $\geq 2$ in $\cH$.
 By induction, the map $\cH\setminus \cZ\to \cY$ extends to $f_\cH:\cH\to \cY$.  Let $g_\cH:\cH\to\cV$ be the composition with the embedding $\cY\subset\cV$.  Let $\cH_i$ be the $i$th order thickening of $\cH$ in $\cX$.  Given a lift $\cH_i\to \cY$ of $f_\cH:\cH\to \cY$, lifts to a morphism $\cH_{i+1}\to \cV$ are unobstructed since the obstruction lies in $H^1(\cH,\Oo_\cH(-i\cH)\otimes g_\cH^*T_{\cV})=0$; the set of such lifts is naturally a torsor under $H^0(\cH_i,\Oo_\cH(-i\cH)\otimes g_\cH^*T_{\cV})$.  Likewise, lifts of $\cH_i\setminus \cZ\to \cY$ to $\cH_{i+1}\setminus \cZ\to\cV$ are a torsor under $H^0(\cH\setminus \cZ,\Oo_\cH(-i\cH)\otimes  g_{\cH\setminus \cZ}^*T_{\cS_v})\cong H^0(\cH,\Oo_\cX(-i\cH)\otimes g^*_\cH T_{\cV})$, since $\Oo_\cH(-i\cH)\otimes  g^*_\cH T_{\cV}$ is locally free and $\codim_\cH\cZ\geq 2$. 
 In particular, there is a lift $\cH_{i+1}\to \cV$ which agrees with the given map $\cH_{i+1}\setminus\cZ\to\cY\to\cV$.  Since the ideal of $\cH_i$ in $\cH_{i+1}$ is locally free (over $\cH$), this lift factors through a morphism $\cH_{i+1}\to\cY$.  Letting $\cX^{\wedge}_\cH$ denote the completion of $\cX$ along $\cH$, it follows that there is a formal map $\cX^\wedge_\cH\to \cY$ whose restriction to $(\cX\setminus \cZ)_{\cH\setminus \cZ}^{\wedge}$ agrees with the completion of $f:\cX\setminus \cZ\to \cY$ along $\cH\setminus \cZ$.  

Let $\cW$ be the closure of the graph of $\cX\setminus \cZ\to\cY$ in $\cX\times\cY$; note that since $\cX'\to\cX$ is proper, it follows that $\cW\to\cX$ is as well.  Since the ideal sheaf of $\cH_i$ in $\cH_{i+1}$ is locally free (over $\cH$) for each $i$, the natural formal map $\cX_\cH^\wedge\to \cX\times \cY$ factors through $\cW$ as it does so over $(\cX\setminus \cZ)_{\cH\setminus \cZ}^\wedge$.  Thus, the projection $\cW\to \cX$ admits a section over $\cX_\cH^\wedge$.  Say the image of $w$ is $x$; then the map $\Spec(\Oo_{\cW,w})\to\Spec(\Oo_{\cX,x})$ admits a formal section.  Since $\cW\to \cX$ is proper and birational, the following lemma implies it is an isomorphism above $x$. 

\begin{lemma}Let $q:\cW\to\cX$ be a proper birational morphism of finite-type integral $\cO$-schemes, $w$ a closed point of $\cW$ and $x=q(w)$.  If the induced map $\Spec(\Oo_{\cW,w})\to \Spec(\Oo_{\cX,x})$ admits a formal section then it is an isomorphism at $w$.
\end{lemma}
\begin{proof}  The claim is \'etale-local on $\cX$.  By Artin approximation \cite[Corollary (2.5)]{artinapprox} (see also \cite[\href{https://stacks.math.columbia.edu/tag/0CAU}{Tag 0CAU}]{stacks-project}), after replacing $\cX$ with an \'etale cover there is an algebraic map $\cX\to \cW$ agreeing with the formal section to first order.  But the composition $\cX\to \cW\to \cX$ agrees with the identity to first order, and therefore is an isomorphism.  We may assume it is in fact the identity, in which case $\cW\to \cX$ admits an algebraic section.  It follows that the short exact sequence of $\Oo_{\cX,x}$-modules
\[0\to\Oo_{\cX,x}\to \Oo_{\cW,w}\to \Oo_{\cW,w}/\Oo_{\cX,x}\to 0\]
splits.  Since $q$ is birational, $\Oo_{\cW,w}/\Oo_{\cX,x}$ is torsion, hence must be zero since $\Oo_{\cW,w}$ is a domain.
\end{proof}

\end{proof}

\begin{proof}[Proof of \Cref{local ICM normal}]By assumption, $\cY$ admits a uniform resolution.  Given smooth $\cX/\cO$ and an $\cO$-admissible $f_F:X\to Y$, by \Cref{extend codim 1 nat} $f_F$ extends in codimension 1 and by \Cref{prop: codim2} the rest of the way. 
    
\end{proof}

\begin{remark}
    The $\dim\cX=2$ step in the proof of \Cref{prop: codim2} in the setting of \Cref{local ICM smooth} follows immediately from the last claim of \Cref{thm:extncodim2}.
\end{remark}

\subsection{Local models of images of period maps}
In this section we give a slightly different formulation of \Cref{local ICM normal} which deduces the existence of the $\bfM^j$ at each level of the resolution from a single global $\bfM$.

\begin{proposition}\label{local ICM normal different}
    Let $F/\Q_p$ be a finite unramified extension with ring of integers $\cO$.  Let $\cY/\cO$ be a normal finite type flat $\cO$-scheme equipped with:
\begin{enumerate}
    \item A uniform stratified resolution $\{\pi^j:\cS^j\to\cY^j\}$ with natural boundary.
    \item A smooth $\cP/\cO$ with a proper dominant $h:\cP\to\cY$ with $\cO_\cY\xrightarrow{\cong} h_*\cO_\cP$ and a Fontaine-Laffaille module $\bfM\in\mathrm{MF}_{[0,w]}(\hat \cP)$ such that:
    \begin{enumerate}
        \item The filtered flat vector bundle $( M,\nabla,F^\bullet M)$ underlying $\bfM$ is pulled back from $(V,\nabla,F^\bullet V)$ on $\hat \cY$.
        \item The Griffiths bundle of $(V,\nabla,F^\bullet V)$ is ample.
        \item The local system $_{et}\V_p$ on $\cP^{rig}$ corresponding to $\bfM$ under the Fontaine--Laffaille correspondence descends to $\cY^{rig}$.
    \end{enumerate}
    \item We finally suppose for each $j$ there is a smooth $\cT^j/\cO$ and a commutative diagram
\[\begin{tikzcd}
\cT^j\ar[d,"q^j",swap]\ar[r,"t^j"]&\cP\ar[dd,"h"]\\
\cS^j\ar[d,"\pi'^j"']&\\
\cY'^j\ar[r]&\cY
\end{tikzcd}\]
where $q^j$ is proper dominant and \'etale over a dense open subset of $\cS^j$. 
\end{enumerate}
    Then if $p>w+1$, $\cY$ is an integral canonical model.
\end{proposition}
Before the proof we need the following:
\begin{lemma}\label{descend extend}
    Let $F/\Q_p$ be a finite unramified extension with ring of integers $\cO$.  Let $\cS,\cT$ be smooth $\cO$-schemes.

    \begin{enumerate}
    \item\label{descend extend 1} If $f:\cT\to\cS$ is finite \'etale of prime to $p$ degree and $_{et}\V_p$ is a $\Z_p$-local system on $\cS^{rig}$ such that $(f^{rig})^*(_{et}\V_p)$ is crystalline, then $_{et}\V_p$ is crystalline.

    \item\label{descend extend 2} Suppose given:
    \begin{enumerate}
        \item $f:\cT\to\cS$ a proper dominant morphism;
        \item $\cU\subset \cS$ a dense open set with preimage $\cV=f^{-1}(\cU)$ such that $f|_\cV:\cV\to \cU$ is \'etale and $\cS\setminus\cU$ is flat over $\cO$;
        \item Fontaine--Laffaille modules $\bfM\in \MF_{[0,w]}(\hat \cT)$ and $\bfN_{\cU}\in \MF_{[0,w]}(\hat \cU)$ as well as an extension $(N,\nabla,F^\bullet N)$ of the underlying filtered flat vector bundle of $\bfN_{\hat\cU}$ to $\hat \cS$;
        \item an isomorphism of filtered flat bundles $\hat f^*(N,\nabla,F^
        \bullet N)\to (M,\nabla,F^\bullet M)$ whose restriction to $\cV$ underlies an isomorphism $\hat f^*\bfN_{\hat\cU}\xrightarrow{\cong}\bfM|_{\hat \cV}$.

    \end{enumerate}Then there is a Fontaine--Laffaille module $\bfN$ on $\hat\cS$ and an isomorphism $\hat f^*\bfN\to\bfM$ extending the above isomorphism. 
    \end{enumerate}
\end{lemma}

\begin{proof}
    The category of Fontaine--Laffaille modules is a full subcategory of the category of local systems, so the descent datum of $f^*_{et}\V_p$ also descends the corresponding Fontaine--Laffaille modules \cite{faltings}.  

    We now prove (2).  The Frobenius morphism on $\bfM$ gives a morphism $\phi_M:\hat f^*F^*(M,\nabla)\to \hat f^*(M,\nabla)$ which is compatible with the one on $\hat f^*(N,\nabla)$, and we have natural solid commutative diagrams
    \begin{equation}\label{ext square}\begin{tikzcd}  \hat f_{*}\hat f^*F^*(N,\nabla)\ar[rr,"\hat f_{*}\phi_M"]&&f_{*}f^*(N,\nabla)&[-30pt]&[-40pt]\hat f_{*}\hat f^*F^*(N,\nabla)|_{\hat \cU}\ar[rr,"\hat f_{*}\phi_M|_{\hat\cU}"]&&\hat f_{*}\hat f^*(N,\nabla)|_{\hat\cU}\\[-15pt]
    &&&\mbox{restricting to}&&&\\[-15pt]
F^*(N,\nabla)\ar[uu]\ar[rr,"\phi",dashed]&&(N,\nabla)\ar[uu]&&F^*(N,\nabla)|_{\hat\cU}\ar[uu]\ar[rr,"\phi_{N_{\hat\cU}}"]&&(N,\nabla)|_{\hat\cU}.\ar[uu]
    \end{tikzcd}\end{equation}
    where in the top row we push-forward just the underlying vector bundles.  Since $\cT,\cS$ are smooth and $f$ is proper dominant, $\cO_{\cS}\to f_{*}\cO_{\cT}$ has torsion-free cokernel.  Indeed, since $\cS,\cT$ are smooth, $\cO_{\cS}\to f_{*}\cO_{\cT}$ is a split summand outside codimension 2, and both sheaves are reflexive.  As $N$ and $F^*N$ are locally free, the vertical arrows in the left square of \eqref{ext square} have torsion-free cokernel as well, and it follows that $\hat f_{*}\phi_M$ restricts to an extension $\phi:F^*(N,\nabla)\xrightarrow{\cong} (N,\nabla)$.  The extension $\phi$ is flat since it is on $\hat\cU$.  
    
    For the strong divisibility condition, note that by \cref{taylor1} $F^*(M,\nabla)$ has a natural flat filtration which is locally given by $\tilde F^i_M=\sum_{0\leq j\leq i}p^{j}\phi^*F^{i-j}$, and $\phi_M$ is divisible by $p^i$ on $\tilde F^i$.  The strong divisibility condition is then equivalent to $p^{-w}\phi_M:\tilde F^w_M\to M$ being an isomorphism. Using the analog of the diagram \cref{ext square} for this morphism, since $p^{-w}\hat f_*\phi_M:\tilde F^w_M\to f_*M$ is an isomorphism and restricts to $p^{-w}\phi_N$ on $\hat \cU$, we conclude that $p^{-1}w\phi_N$ is an isomorphism. 
\end{proof}

\begin{proof}[Proof of \Cref{local ICM normal different}]
By assumption, for each $j$, we have a Fontaine--Laffaille module $\hat t^{j*}\bfM$ on $\hat \cT^j$ whose corresponding local system is pulled back from $\cY^{rig}$.  Moreover, the underlying filtered flat vector bundle of $\hat t^{j*}\bfM$ is pulled back from $(V,\nabla,F^\bullet V)$ on $\hat \cY$.  Applying part \ref{descend extend 1} of \Cref{descend extend} to the open set $\cU^j$ of $ \cS^j$ over which $q^j$ is finite \'etale, we conclude that there is a Fontaine--Laffaille module $\bfN_{\hat \cU^j}$ which pulls back to $(\hat t^{j*}\bfM)|_{(\hat q^j)^{-1}(\hat \cU^j)}$.  Applying part \ref{descend extend 2}, we find that $\bfN_{\hat \cU^j}$ extends to a Fontaine--Laffaille module $ \bfN^j$ which pulls back to $\hat t^{j*}\bfM$.  In particular, we have $\hat q^{j*}(N^j,F^\bullet N^j)\cong \hat q^{j*}\hat \pi^{j*}(V,F^\bullet V)$.  Thus, all the conditions of \Cref{local ICM normal} are satisfied.
\end{proof}

\subsection{Making period maps proper}
We will need a result allowing us to extend geometric Fontaine--Laffaille modules when the underlying filtered flat vector bundle extends.
\begin{proposition}\label{make period map proper}
    Let $F/\Q_p$ be a finite unramified extension with ring of integers $\cO$.  Let $\cP/\cO$ be smooth with a partial log smooth compactification $(\cP',\cD)/\cO$.  Let $f:\cZ\to \cP$ be a smooth projective morphism with geometrically connected fibers of relative dimension $< p$.  Let $\bfM=H^k_{DR}(\cZ/\cP)$ be the Fontaine--Laffaille module with underlying filtered flat vector bundle the De Rham cohomology of $\cZ/\cP$.  Assume:    
    \begin{enumerate}
    \def\cC{\mathcal{C}}
    \item There is a generically finite proper morphism $g:(\cQ,\cE)\to(\cP',\cD)$ from a log smooth $(\cQ,\Delta)/\cO$ and a semistable morphism $(\cW,\Delta)\to (\cQ,\cE)$ from a log smooth $(\cW,\Delta)/\cO$ extending the base-change of $f$ (possibly after shrinking $\cP$).
    \item The filtered flat vector bundle $(M,\nabla,F^\bullet M)$ extends to $(M',\nabla',F^\bullet M')$ on $\cP'$.
    \item $g'^*(M',\nabla',F^\bullet M')\cong H^k_{DR}((\cW, \Delta)/(\cP',\cD))$ via an isomorphism extending the natural one, where $H^k_{DR}((\cW, \Delta)/(\cP',\cD))$ is the relative log De Rham cohomology.
    \end{enumerate}
    Then $(M',\nabla',F^\bullet M')$ underlies an extension $\bfM'$ of $\bfM$ to $\cP'$ as a Fontaine--Laffaille module.
\end{proposition}
\begin{proof}According to \cite[Thm 6.2]{faltings}, the completion of the filtered flat bundle $H^k_{DR}((\cW, \Delta)/(\cP',\cD))$ naturally underlies a log Fontaine--Laffaile module \cite[\S2i]{faltings}.  By assumption the connection has trivial residue, so it follows it is in fact a Fontaine--Laffaille module.  Now apply \Cref{descend extend}(2).
\end{proof}

\section{Existence of integral canonical models}\label{sec: canonical models}
In this section, we deduce the existence of integral canonical models for sufficiently large primes.

\subsection{Proof of \Cref{shimura main}}\label{sec: Shimura setup}
In this section we show that models of Shimura varieties satisfy the conditions of \Cref{local ICM smooth} for large primes, thereby proving Theorem \ref{shimura main}. Let $S=S_K(G,X)$ be a Shimura variety. 
We let $E$ denote the reflex field.
$\ShimK$ carries a flat principal $G$-bundle $\mathcal{E}/S$, which admits a canonical model over $E$. A choice of Hodge co-character defines a parabolic subgroup $P\subset G$, and gives a parabolic bundle $\mathcal{P}/S$, along with a canonical morphism $\mathcal{P} \rightarrow \mathcal{E}$, where all this data descends to $E_G$. We note that $\mathcal{P}$ is not flat for the connection on $\mathcal{E}$. 

Let $V$ denote a faithful rational representation of $G$, and $\V\subset V$ denote a lattice. The choice of $\V$ specifies a compact open subgroup of $ G(\A_f)$ by considering the stabilizers of $\V_{\hat\Z}$. We choose the level structure $K$ to be a subgroup of this compact open subgroup, such that it is neat, and such that it acts trivially on $\V/3\V$. Corresponding to $V$ there is a family of $\Q_\ell$ local systems $_{\et}V_\ell$ with $\Z_\ell$ sub-systems $_{\et}\V_\ell$. Moreover, there is a filtered flat vector bundle $_{\dR}V$ over $\ShimK$ -- indeed we have that $_{\dR}V$ is defined by $\mathcal{E} \times_G V$. The filtration is induced by the fact that $_{\dR}V$ is also canonically isomorphic to $\mathcal{P}\times_{G_{E}} V_{E}$. Now, let $p$ be a prime and let $v \mid p$ be a prime of $E$. 

By choosing $N$ to be a large enough integer, we may assume that $\ShimK$ spreads out to a smooth integral model $\integralShimK/\cO_{E}[1/N]$ such that:
\begin{itemize}
    \item $G$ has a reductive model $\mathcal{G} / \Z[1/N]$. There are elements $s_{\alpha} \in \V[1/N]^{\otimes}$ whose point-wise stabilizers define the group $\mathcal{G}$.
    \item For every $p\nmid N$, we have $K =\mathcal{G}(\Z_p) \cdot K^p$, where $K^p$ is the prime-to-$p$ level structure.
    \item $P\subset G_{E}$ is induced by a parabolic subgroup of $\mathcal{G}$ defined over $\cO_{E}[1/N]$.
    \item The map $\mathcal{P} \rightarrow \mathcal{E}$ is induced by a map $\mathscr{P} \rightarrow \mathscr{E}$ over $\cS$, where $\mathscr{E}$ is a flat principal $\mathcal{G}$ bundle and $\mathscr{P}$ is a principal bundle for the parabolic of $\mathcal{G}$ defined over $\cO_{E}[1/N]$.
\end{itemize}
Note that this gives us an integral model $_{\dR}\V/\integralShimK$ for the filtered flat bundle $_{\dR}V$ (where the flat bundle is defined as $\mathscr{E} \times_{\mathcal{G}} \V[\frac{1}{N}]$, and the filtration is induced by $\mathscr{P} \rightarrow \mathscr{E}$). 

\begin{lemma}
    After increasing $N$ we have the following.
    \begin{itemize}
    \item $\integralShimK$ admits a log-smooth compactification over $\cO_{E}[1/N]$.
    \item The local systems $_{\et}\V_{\ell}$ extend to $\integralShimK[1/\ell]$.
    \item For each finite place $v$ of $\cO_E[1/N]$, $_{\et}\V_p(m)$ is crystalline on $\integralShimK_{E_v}$ for $m$ sufficiently small so that $V(m)$ has positive Hodge degrees.
    \item For each finite place $v$ of $\cO_E[1/N]$, let $_{\cris}\V/\integralShimK_{\cO_{E_v}}$ denote the Fontaine-Laffaille module corresponding to $_{\et}\V_p$ under the Fontaine-Laffaille correspondence.  Then the Griffiths bundle of $_{\cris}\V$ is ample.
\end{itemize}
\end{lemma}
\begin{proof}
The first condition follows by choosing a log smooth compactification over the generic fiber.

For the second condition, we use the following result of \cite[Thm 1.3]{KP}:

 \begin{theorem}[Klevdal-Patrikis]
There is an integral model $\cS$ of $S_K(G,X)$ over $\Z[1/N]$ for some $N$ such that for each $\ell$
the local systems $_{\et}\V_{\ell}$ all extend to $\cS[1/\ell]$.

\end{theorem}

The third condition is \cite[Thm 7.1]{PST}

    The final condition follows because the same is true for the filtered flat bundle $_{\dR}V$ over the generic point. 
\end{proof}
It follows from \Cref{local ICM smooth} that $\integralShimK_{\cO_{E_v}}$ is an integral canonical model.  It remains to argue that $\integralShimK_{\cO_{E_v}}$ is the unique model admitting a log smooth compactification.  Let $\cS$ be another such model. The extension property yields that the identity map over the generic fiber extends to $f:\cS\to \integralShimK_{\cO_{E_v}}$, and by purity of the branch locus this map is an open immersion.  Let $\cW\subset\bar\cS\times \barintegralShimK_{\cO_{E_v}}$ be the closure of the graph of $f$.  Over the generic point, the inverse image of the boundary of $\bar S$ in $W$ must map to the boundary of $\barintegralShimK_{\cO_{E_v}}$, since every boundary divisor has infinite local monodromy, and the same is therefore true for $\cW$.  Thus, the extension is an isomorphism. \qed

\subsection{Proof of \Cref{thm ICM hodge}}\label{sec: main theorem for geometric period map images}
Suppose $F$ is a number field, $P/F$ a smooth variety, and $f:Z\to P$ a smooth projective family.  Let $h_F:P'\to Y$ be the Stein factorization of the (compactification) of the period map associated to $\W_\Z:=Rf_*\Z_{Z^\an}$ as described in \ref{hodge sect}.  Then $h_F$ is defined over $F$.  By spreading out over $\cO:=\cO_F[1/n]$ and inverting finitely many primes, we may assume we have smooth $\cP,\cZ/\cO$, a log smooth partial compactification $(\cP',\cD)/\cO$ of $\cP$, and $\cY/\cO$ normal finite type flat as well as:
\begin{enumerate}
\item $f:\cZ\to \cP$ smooth projective.
\item $h:\cP'\to\cY$ proper with $\cO_\cY\xrightarrow{\cong}h_*\cO_{\cP'}$.
\item $\cY$ admits a uniform stratified resolution $\{\pi^j:\cS^j\to\cY^j\}$.
\item For each $j$ there is a smooth $\cT^j/\cO$ and a commutative diagram
\[\begin{tikzcd}
\cT^j\ar[d,"q^j",swap]\ar[r,"t^j"]&\cP'\ar[dd,"h"]\\
\cS^j\ar[d,"\pi^j"']&\\
\cY'^j\ar[r]&\cY
\end{tikzcd}\]
where $q^j$ is proper dominant and \'etale over a dense open subset of $\cS^j$. 
\item By \cite{semistab_reduction} there is a semistable model $(\cW,\Delta)\to(\cQ,\cE)$ for the family $f$ up to a smooth alteration $g:(\cQ,\cE)\to(\cP',\cD)$.
\def\cV{\mathcal{V}}
\item The filtered flat vector bundle $(V,\nabla,F^\bullet V)$ on $Y$ extends to $(\cV,\nabla,F^\bullet \cV)$ on $\cY$, and its pullback to $\cP$ is identified with the filtered flat vector bundle underlying $H^k_{DR}(\cZ/\cP)$.  Moreover, $(\cV,\nabla,F^\bullet\cV)$ has ample Griffiths bundle by \Cref{thm:bbt}.
\item By \Cref{descend to integral}, the local system $H^k_{et}(\cZ/\cP[1/\ell],\Z_\ell)$ extends to $\cP'$ and descends to $_{et}\V_\ell$ on $\cY[1/\ell]$.  By \Cref{griffiths extension} and \Cref{faithful => natural}, it follows that the stratified resolution has natural boundary. 
\end{enumerate}
This yields part \ref{thm ICM hodge p1a}, and both \ref{thm ICM hodge p2} and \ref{thm ICM hodge p3} follow from \ref{thm ICM hodge p1b} by \Cref{uniqueness among models} and \Cref{extend maps}.  Part \ref{thm ICM hodge p4} is clear from construction.  Thus it remains to show the extension property in \ref{thm ICM hodge p1b}, and since by \Cref{griffiths extension} we have the extension in characterstic 0, it suffices to show $\cY_v$ is an integral canonical model for each finite place $v$.  By \Cref{local ICM normal} we must show for each finite place $v$ of $\cO$ that there is a Fontaine--Laffaille module $\bfM$ on $\cP'_v$ with the required properties.  The filtered flat vector bundle underlying $H^k_{DR}(\cZ_v/\cP_v)$ extends to $\cP'_v$, and so by \Cref{make period map proper} it extends as a Fontaine--Laffaille module whose underlying filtered flat vector bundle is the pullback of $(\cV,\nabla,F^\bullet \cV)_{\cO_v}$ on $\cY_v$.\qed

\section{Tate-semisimplicity and special points}\label{sec: Tate semisimplicity}
The aim of this section is to first prove that prime-to-$p$ Hecke correspondences extend to the integral model over $\cO_{E,v}$ where $v\mid p$ and use this to prove that the Frobenius endomorphism on $_\et\V_{\ell,x}$ is semisimple for every $x\in \integralShimK(\F_q)$. The main theorem of this section is:

\begin{theorem}\label{thm:Tatess}

Let $p\not\mid N$ be a prime, and $x\in \integralShimK(\F_q)$ where $k$ is a finite field of char. $p$. Then $_{\et}\frob_{\ell,x}$ acts semisimply on $V_{\ell,x}$. 

\end{theorem}

This is a special case of a folklore conjecture that posits that the Frobenii at closed points are semisimple for any irreducible local system defined over a smooth $\F_q$-variety. We will then use this to show that every $\overline{\F}_p$-point of $\integralShimK$ is special. As the proof of Theorem \ref{thm:Tatess} strongly uses Hecke correspondences, we will start this section with a discussion on Hecke correspondences in characteristic zero and their extension to $\integralShimK$. We will then reformulate the notion of special points in terms of the Hecke action, and then prove Theorem \ref{thm:Tatess} and use that to deduce that every $\overline{\F}_p$-point of $\integralShimK$ is special. 


\subsubsection{Extension of Hecke correspondences}

Let $\V$ be a lattice in a faithful representation $V$ of $G$, such that $\V\otimes_\Z \hat{Z}$ stabilized by $K$.
Let $K(n)$ denote the kernel $K \ra \Aut(\V/n\V)$, and consider the congruence covers $S_n:=S_{K(\ell^n)}(G,X)$. We focus on the towers $S_{\ell^n}$ for a fixed prime $\ell$. Over the generic fiber this is a (possibly disconnected) Galois cover of $S=S_1$.  By taking the total space of the local systems $\V/\ell^n\V$ we obtain smooth integral models $\cS_n$ with log-smooth compactifications over the base, that are Galois over $\integralShimK=\cS_1$.


We recall some basics about Hecke correspondences. Let $h\in G(\Z[1/\ell])$.  Let $\ell^n$ be such that $h^{-1}K_{\ell}(\ell^n)h\subset K_\ell$. Then there are well defined maps 
$\pi_h:S_{\ell^{n}}\ra S_1$ for large enough $n$.  Let $\pi:S_{\ell^n}\ra S_1$ denote the standard projection. 

Then $\pi^{-1}x$ is corresponds to (a subset of) mod $\ell^n$-equivalence classes of trivializations of $t:\V_x\cong \V $, and $\pi_h(x_1)=\pi_h(x_2)$ for $\pi(x_1)=\pi(x_2)=x$ if and only if
$t_{x_2}t_{x_1}^{-1} \in K_\ell$.

\begin{theorem}\label{thm:heckeextends}
For $h\in G(\Q_\ell)$ the map $\pi_h:S_{\ell^n}\ra S$ - which exists for large enough $n$ -  extends as an etale map over $\Z[1/\ell N]$
\end{theorem}

\begin{proof}

This is an immediate consequence of Theorem \ref{shimura main}.

\end{proof}
With this in hand, we have: 

\begin{definition}
    The $\ell$-power Hecke correspondence $\tau_h$ associated to $h$ on $\integralShimK/\cO_{E}[1/\ell N]$  is defined as $\tau_h(x) = \pi_h(\pi^{-1}(x))$. 
\end{definition}

Let $p\nmid \ell N $ be a prime, let $F$ be a $p$-adic local field with $\cO_F \subset F$ its ring of integers. Let $y\in\integralShimK(\cO_F)$. We then have that $\tau_h(y) = \{y_1 \hdots y_n \}$ for some integers $n$, where $y_i \in \integralShimK(\cO_{F_i})$ where $F_i/F$ are finite unramified extensions. There is a canonical Galois action of $\pi_1(\cO_F)$ on the set $\{ y_1 \hdots y_n\}$ induced by the action of $\pi_1(\cO_F)$ on $_{\et}\V_{\ell,x}$. Further, we have that the field of definition $F_i$ of $y_i$ is just the extension corresponding to the stabilizer of $y_i$ in $\pi_1(\cO_F)$. Finally, the Galois action of $\pi_1(\cO_{F_i})$ on $_{\et}\V_{\ell, y_i}$ is induced by the Galois action of $\pi_1(\cO_F)$ on $_{\et}V_{\ell,x} = _{\et}\V_{\ell,x}\otimes \Q_{\ell}$ by considering $_{\et}\V_{\ell,y_i} = {t_i^{-1}} h _{\et}\V_{\ell} \subset _{\et}V_{\ell,x}$ (here, $t_i$ is the choice of mod $\ell^n$-trivialiazation corresponding to the point $y_i$). 
This description also holds with $\cO_F$ replaced with finite fields, as $\pi_1(\cO_F) = \pi_1(\F_q)$ where $\F_q$ is the residue field of $\cO_F$. 

This discussion has the following immediate consequence.
\begin{lemma}\label{Lemma Hecke description}
    Let $y \in \integralShimK(\cO_F)$ (resp. $\integralShimK(\F_q)$). Let $\sigma \in \pi_1(\cO_F)$ (resp. $\pi_1(\F_q)$) denote the geometric Frobenius endomorphism, and let $_{\et}\varphi_{\ell,x}$ denote the automorphism of $_{\et}\V_{\ell,x}$ induced by $\sigma$. Let $h\in G(\Q_{\ell})$ be some element, and let $\tau_h(y) = \{y_1 \hdots y_n \}$. We have $y_i$ is in $\integralShimK(\cO_F)$ (resp. $\integralShimK(\F_q)$) if $\sigma (t_i^{-1} h\V_{\ell}) = t_i^{-1} h \V_{\ell}$. Further, if we were to identify $\V_{\ell}$ with $_{\et}\V_{\ell,y}$ using $t_i$, we obtain the identification $h\V_{\ell} \cong _{\et}\V_{\ell,y_i}$, with the Frobenius at $y_i$ acting as $h^{-1} _{\et}\varphi_{\ell,x} h$. 
\end{lemma}


%

\subsection{Special points}

\begin{proposition}\label{prop:charecterization of special points}
    Let $x \in \ShimK(\bbC)$. The following are equivalent: 
    \begin{enumerate}
        \item $x$ is a special point.

        \item For infinitely many primes $\ell$ at which $G$ is split, there exists a finite set $S:=\{x = x_1, \hdots, x_n \} \subset \ShimK(\bbC)$ and a maximal split torus $T\subset G$ defined over $\Q_{\ell}$ such that $\tau_h(x) \cap S \neq \emptyset$ for every $h\in T(\Q_{\ell})$.
    \end{enumerate} 
\end{proposition}

\begin{proof}
Assume 1. Then there is a $\Q$-torus $T\subset G$ and a cocharacter $h:\bbS^1\ra T_\R$ which lands in $X$. By enlarging $T$ assume that $T$ is maximal. Then conjugating by elements of $T$ does not change $h$, and the statement follows.

Now assume 2. Let $y=y_1,\dots,y_n \in G(\A)/K$ be lifts of the $x_i$. Without loss of generality we may assume that for every $i=1,\dots,n$ there is an element $c_i\in G(\Q)$ such that $c_i\cdot y=y_i$. Moreover, we let $z_i:=y_{i,\R}\in X.$ 

Now let $H\subset G_\Q$ denote the largest $\Q$-subgroup fixing $z=z_1$. Note that $H(\R)$ is compact and therefore $H$ is reductive. Let $\ell$ be a sufficiently large prime such that $G,H$ are both split, $K_\ell$ and $K_\ell\cap H(\Q_\ell)$ are hyperspecial, the torus $T$ exists as in assumption 2, and all the $y_i$ have component $1\cdot K_\ell$ at $\ell$. Then $\forall h\in T(\Q_\ell), \tau_{h^{-1}}(y)$ consists of the union of $h_i\cdot y$ where the multiplication takes place in the $\ell$'th co-ordinate only, and $h_i$ ranges over elements of $K_\ell h^{-1} K_\ell$. Therefore, there is an element $\gamma\in G(\Q)$ such that 

\begin{itemize}
    \item $\gamma_\ell\in K_\ell h K_\ell$
    \item $\gamma \cdot z \in \{z_1,\dots,z_n\}$
\end{itemize}

The second condition is that for some $i, \gamma\in c_i\cdot H$, and thus we see that
$$\bigcup_{i=1}^m c_i\cdot H(\Q_\ell) \twoheadrightarrow K_\ell\backslash G(\Q_\ell)/K_\ell.$$

Let $Z\subset G$ be a maximal, split $\Q_\ell$-torus such that $Z':=Z\cap H$ is also a maximal, split $\Q_\ell$ torus of $H$. Let $W,W'     $ be the Weyl groups of $(G,Z), (H,Z')$. By \cite[Pg. 51] {tits}, we have 
$$K_\ell\backslash G(\Q_\ell)/K_\ell\cong X_*(Z)/W$$ and similarly
$$ (K_\ell\cap H)\backslash H(\Q_\ell)/(K_\ell\cap H)\cong X_*(Z')/W'.$$

Let ${P_H}^+\subset P^+$ be the coroots in a positive Weyl Chamber for $Z',Z$ respectively. Let us write $\bigcup_{i=1}^m K_{\ell}  c_i C_{\ell}$ as %
$\bigcup_{\lambda\in S} K_{\ell}\lambda(\ell)K_{\ell} $ for $S\subset P^+$. Then 
\begin{align*}
G(\Q_{\ell}) &= \bigcup_{i=1}^m K_\ell c_i\cdot H(\Q_\ell)K_\ell \subset \left(\bigcup_{i=1}^m K_\ell c_iK_\ell\right) \cdot K_\ell H(\Q_\ell)K_\ell\\ &\subset \bigcup_{\lambda\in S} K_\ell\lambda(\ell)K_\ell \bigcup_{\lambda'\in P_H^+} K_\ell\lambda'(\ell)K_\ell.\\
\end{align*}
Note that the containments must in fact be equalities.
By \cite[Pg. 148]{tits} we see that this latter product contains $\mu(\ell)$ for $\mu\in P^+$ only if $\mu$ is dominated by some root in $S+P_H^+$. But $S$ is finite, $P_H^+$ is contained in a sublattice of $X_*(Z)$, so this is only possible if $Z'=Z$. 

 Thus $H$ has the same rank as $G$. Thus $z$ has a stabilizer containing a maximal $\Q$-torus, and therefore $x$ is a special point, as desired.

\end{proof}

Let $R$ denote a ring of characteristic $p$. In the case of Shimura varieties of Hodge type, one defines an $R$-valued point to be special if the abelian variety associated to the point has complex multiplication. Guided by Proposition \ref{prop:charecterization of special points}, we make the following definition. 
\begin{definition}\label{def:special points}
    Let $R$ be a ring of characteristic $p$. We say that $x\in \integralShimK(R)$ is special if there exists a finite set $x = x_1, x_2, \hdots x_n \in \integralShimK$, and split maximal tori $T_{\ell}\subset G_{\Q_\ell}$ for infinitely many primes $\ell$ such that $\tau_h(x) \cap \{x_1, \hdots, x_n \} \neq \emptyset$ for every $h\in T(\Q_{\ell})$. 
\end{definition}

We will end this section by proving Theorem \ref{thm:Tatess} and using it to prove that every $\overline{\F}_p$-valued point of $\integralShimK$ is special as per our definition.

\subsection{Tate semisimplicity}

We will now prove Theorem \ref{thm:Tatess}.

\begin{proof}
    
We first fix a $G$-equivariant identification $_{\et}\V_{\ell,x} \cong \V_{\ell}$. Suppose that $_{\et}\frob_{\ell,x}$ doesn't act semisimply -- let $(_{\et}\frob_{x,\ell}) = \frob_{\sss}\cdot\frob_{\unip}$. Note that the Jordan decomposition implies that $\frob_{\sss},\frob_{\unip}$ are elements of $G(\Q_\ell)$. A sufficiently $\ell$-divisible power of $\frob_{\unip}$ will preserve $\V_{\ell,x}$ and hence by increasing the field $\F_q$ we may and do assume $\frob_{\sss}, \frob_{\unip}$ lies in $K_{\ell}$. 

We will find hecke operators in $G(\Q_\ell)$ which preserve $\F_q$-rationality, and therefore which eventually give us that our point is stable under that appropriate operator, which we will use to deduce semisimplicity. In order to preserve $\F_q$-rationality, we pick our Hecke operator as follows. Let $H\subset G$ denote the centralizer of $\frob_{\sss}$ -- it contains $\frob_x$ and $\frob_{\unip}$. Note that $H$ is semisimple by \cite{Conrad:Reductive Schemes}. Using the Jacobson-Morozov lemma, we can find an element $h\in H(\Q_{\ell})$ such that $h^{-1}\frob_{\unip} h = \frob_{\unip}^{\ell^2}$. Let $\tau_h$ denote the corresponding Hecke correspondence. One of the points of $\tau_h(x)$ (which we call $x_1$) will correspond to the lattice $h \V_{\ell}\subset V_{\ell}\otimes\Q_\ell $, and as this lattice is stable under the action of $\frob$, the point will also be defined over $\F_q$ by Lemma \ref{Lemma Hecke description}. The action of $\frob_{x_1}$ is just $\frob_x$ on the lattice $h\V_{\ell} \subset V_{\ell} $, again by Lemma \ref{Lemma Hecke description}. Therefore, we have that $\frob_{x_1} = h^{-1}\frob_x h  =  \frob_{\sss}\frob_{\unip}^{\ell^2} $ by construction. Therefore, $\frob_{x_1,\unip} - 1$ is strictly more divisible by $\ell$ then $\frob_{x,\unip}-1$, and therefore $x\neq x_1$.

Iterating this process over and over, we get a sequence of mutually distinct $k$-rational points. This is contradiction arising from our assumption that $\frob_{\unip}$ is nontrivial.

\end{proof}

Theorem \ref{thm:Tatess} and its proof has the following immediate consequence. 
\begin{corollary}\label{cor:everypointCM}
Every $x\in \integralShimK(\overline{\F}_p)$ is special.
\end{corollary}

\begin{proof}
    It suffices to prove that the image of $x$ in the adjoint Shimura variety is special, and so we suppose that $G$ is an adjoint group. The point $x$ is defined over some $\F_q$.  Let $H$ denote the centralizer of $_{\ell}\phi_x$, where $_{\ell}\phi_x$ is the $\F_q$-Frobenius at $x$ acting on $_{\et,\ell}\V_{x}$. The argument in the proof of Theorem \ref{thm:Tatess} yields that $\tau_h(x) \cap \integralShimK(\F_q) \neq \emptyset$ for every $h\in H(\Q_{\ell})$. By Theorem \ref{thm:Tatess}, we have that $H$ must contain a maximal torus of $ G$ defined over $\Q_{\ell}$. Therefore, if we were to set $S = \{x = x_1, \hdots x_n \} = \integralShimK(\F_q)$ and $T$ to be some maximal torus of $H$ defined over $\Q_{\ell}$, we would have that $\tau_h(x)\cap S \neq \emptyset$ for every $h\in T(\Q_{\ell})$. Therefore, it suffices to prove that there exist infinitely many $\ell$ such that some power of $_{\ell}\phi_x$ is contained in a split maximal $\Q_{\ell}$-torus of $G$. 

    By the main theorem of \cite{StefanChristian}, the conjugacy class of $_{\ell}\phi_x$ is $\Q$-rational, and does not depend on $\ell$. Therefore, we may choose $\ell$ such that $G_{\Q_{\ell}}$ is split, and such that the characteristic polynomial of $_{\ell}{\phi_x}$ on some faithful representation of $G$ splits over $\Q_{\ell}$. By replacing $_\ell \phi_x$ by a power (i.e. replacing $q$ by a power), we may also assume that the roots of this characteristic polynomial do not differ by a non-trivial root of unity. It follows that $_{\ell}\phi_x$ is contained in split torus. However, in a quasi-split (and therefore split) group, every split torus is contained in a maximal split torus.

    \end{proof}

\section{Complete local rings for large primes}\label{sec: complete local rings for large primes}
Let notation be as in Section \ref{sec: Shimura setup}, and let $p\nmid N$ denote a prime. As we will exclusively work over a $p$-adic field, we let $G$ denote $\mathcal{G}_{\Z_p}$. Throughout this section, the level structure at $p$ is $G(\Z_p)$. 

The goal of this section is to fully describe the Fontaine-Laffaille data restricted to complete local rings of $\ShimK$ at $\overline{\F}_p$-points $x$, and to describe the Fontaine-Laffaille data at $W$-valued points of $\ShimK$ in terms of picking a filtration on an appropriate $F$-crystal. For such a point $x$, we let $\integralShimK^x$ denote the formal neighbourhood of $\integralShimK$ at $x$. Let $R^x$ denote the complete local ring at $x$, and let $_{\cris}\V^x$ denote the restriction of $_{\cris}\V$ to $\integralShimK^x$. Similarly, we let $_{\cris}\V_{x}$ denote the restriction of $_{\cris}\V$ to any $W$-point of $\integralShimK$. Let $_{\cris}\V_x$ denote the induced $F$-crystal at $x\in \integralShimK(\F)$. We will treat $_{\cris}\V_x$ as a free $W$-module equipped with a semi-linear Frobenius endomorphism. We let $\varphi_u$ denote Frobenius on $_{\cris}\V^x$, and $\varphi_{\tilde{x}}$ denote Frobenius on $_{\cris}\V_{\tilde{x}}$.
\subsection{The Fontaine-Laffaille module on complete local rings}\label{sec: complete local rings}
    \subsubsection{F-crystals with $G$-structure and $G$-split filtrations.}\label{subsection G-structure}
We refer to \cite[Section 2]{ShankarZhou} for the notion of an $F$-crystal with $G$ structure. 
 Recall that we have tensors $s_{\alpha}\in \V_{\frac{1}{N}}^{\otimes}$. These tensors induce global sections of the $p$-adic etale local system $_{\et}\V_p^{\otimes}$ (defined over $\integralShimK/\Oo_E[\frac{1}{pN}]$) -- we will refer to them as etale tensors $_{\et}s_{\alpha,p} \in _{\et}\V_p^{\otimes}$. 

 Via the Riemann-Hilbert correspondence, the tensors $s_{\alpha}$ give rise to flat sections $_{\dR}s_{\alpha} \in _{\dR}V^{\otimes}$. We have the following proposition:
 \begin{proposition}\label{prop: dR tensors integral and defined over reflex }
     The tensors $_{\dR}s_{\alpha}$ descend to integral tensors in $_{\dR}\V^{\otimes}$ over $\cO_E[\frac{1}{N}]$.
 \end{proposition}
 \begin{proof}
     This follows directly from the existence of the integral model $\mathscr{E}/\cO_E[\frac{1}{N}]$ of the flat principal $\mathcal{G}$-bundle. Then we have the isomorphism \begin{equation}\label{eqprincipalGbdle}
         _{\dR}\V^{\otimes} \simeq \mathscr{E}\times_{\mathcal{G}} \V ^{\otimes}
     \end{equation} over $\cO_E[\frac{1}{N}]$. Using this isomorphism, the tenors $s_{\alpha}$ induce canonical flat sections of $_{\dR}\V^{\otimes}$. These are precisely the flat sections $_{\dR}s_{\alpha}$, and the proposition follows. 
     
 \end{proof}

Let $v$ denote some prime of $E$ above $p$, and consider the $p$-adic etale local system $_{\et}\V_p / \ShimK/E_{v}$. As in \cite[Section 1.3.3]{Kisinintegral}, we may apply the Fontaine-Laffaille functor to the tensors $_{\et}s_{\alpha,p}$ to obtain tensors $_{\cris}s_{\alpha} \in {_{\cris} \V^{\otimes}}$, which by construction are fixed by Frobenius and contained in $\Fil^0$. 

Recall that the integral structures on $_{\dR}V/\ShimK_{E_v}$ induced by $_{\cris}\V$ and $_{\dR}\V$ are the same. We have the following proposition: 
\begin{proposition}\label{prop: compatibility de rham crys tensors}
    Under the identifications above, we have $_{\cris}s_{\alpha} = {_{\dR}s_{\alpha}}$.
\end{proposition}
 \begin{proof}
     As both sets of tensors are flat, it suffices to prove that the two sets of tensors agree at some point -- we will show that they agree at a special point. Indeed, at a special point $y$, all the data $_{?}\V$ arise from the cohomology of a CM abelian variety. Similar to \cite[Section 2.2.2]{Kisinintegral}, the equalities follow from the work of Blasius-Wintenberger \cite{Blasius}. 
 \end{proof}

\begin{corollary}\label{cor:Gstructure}
    There is an (non-canonical\footnote{The non-canonicity is because any choice of isomorphism can be pre-composed by some element of $G(R^x)$.}) isomorphism of $R^x$-modules $\iota^x:\V_p \otimes R^x \rightarrow {_{\cris}\V^x}$ which sends the tensors $s_{\alpha}$ to $_{\cris}s_{\alpha}$.
\end{corollary}
 \begin{proof}
     This follows directly from the ismorphism in Equation \eqref{eqprincipalGbdle} and the fact that $\mathscr{E}|_{{\integralShimK}^x}$ is trivializable (Note that the connection however is \emph{not} trivializable as a principal bundle). 
 \end{proof}

Now, let $\tilde{x} \in \integralShimK(W)$ denote a $W$-valued lift of $x$. Let $_{\cris}\V_{\tilde{x}}$ denote the Fontaine-Laffaille module pulled back to $\tilde{x}$. Corollary \ref{cor:Gstructure} has the following immediate implication: 
\begin{corollary}\label{moreGstructure}
    There are isomorphisms $\iota_{\tilde{x}}:\V_p\otimes W \rightarrow {_{\cris}\V_{\tilde{x}}}$ and $_{\cris}\V_{\tilde{x}}\otimes R^x \xrightarrow{\iota} {_{\cris}\V^x}$ that respect the tensors $s_{\alpha}$, and such that $\iota\circ \iota_{\tilde{x}}=\iota^x$.
\end{corollary}
 
We will now prove that the second isomorphism in Corollary \ref{moreGstructure} can be chosen to respect filtrations. 

\begin{proposition}\label{equalfilt}
The isomorphism $\iota$ in Corollary \ref{moreGstructure} may be chosen such that $\Fildot(_{\cris}\V_{\tilde{x}})\otimes R^x$ maps isomorphically onto $\Fildot(_{\cris}\V^x)$. 
\end{proposition}
\begin{proof}

Consider the subgroups of $\GL(_{\cris}\V_{\tilde{x}})$ and $\GL(_{\cris}\V^{x})$ that fix the tensors $s_\alpha$ and the filtrations $\Fildot$. The existence of the integral model of $\mathscr{P}/\integralShimK$ implies that these are parabolic subgroups  $P_W \subset G_W$ and $P_{R^x}\subset G_{R^x}$ respectively -- the identification $\iota_{\tilde{x}}$ and $\iota^x$ in Corollaries \ref{cor:Gstructure} and \ref{moreGstructure} allow us to think of the parabolic subgroups as subgroups of $G_W$ and $G_{R^{x}}$. 
Further, $P_{R^x}$ clearly specializes to $P_W$ via the point $\tilde{x}$. Consider now the two parabolic subgroups $P_{R^x}$ and $P_W\otimes R^x$ of $F_{R^x}$. These two subgroups are equal (and therefore conjugate) at the unique closed point. By \cite[Corollary 5.2.7]{Conrad:ReductiveSchemes}, the two parabolics are $\mathcal{G}(R^x)$-conjugate. Conjugating $\iota$ by this element still respects the tensors (which are fixed by elements of $G$) and now respects the filtrations by construction. The proposition follows. 
\end{proof}

Finally, we have that $\Fildot$ is induced by a $G$-split co-character. In other words, we have: 
\begin{corollary}\label{cor:filtisGsplit}
    The filtration on  $\Fildot(_{\cris}\V_{\tilde{x}})$ is induced by a $G_W$-valued co-character. 
\end{corollary}
\begin{proof}
   We will let $\tilde{x}$ also denote the $W[1/p]$-point of $\integralShimK$. By \cite[Lemma 1.1.4]{Kisinintegral}, it suffices to prove that the filtration on $_{\dR}V_{\tilde{x}}$ is induced by a $G_{W[1/p]}$-valued co-character. Further, we have (by \cite{DLLZ}) that $_{\dR}V_{\tilde{x}}$ and $D_{\dR}(_{\et}V_{p,\tilde{x}})$ are isomorphic as filtered vector spaces. The fact that the filtration on $D_{\dR}(_{\et}V_{p,\tilde{x}})$ is induced by a $G$-valued co-character now follows from \cite[Lemma 1.4.5]{Kisinintegral}.
\end{proof}

\subsubsection{Explicit coordinates}

We pick coordinates on $R^x$ and write it as $W[[x_1, \hdots, x_n]]$, such that the point $\tilde{x}$ is given by setting all the $x_i = 0$. We choose a lift of Frobenius $\sigma_u$ on $R^x$ that extends the usual Frobenius $\sigma$ on $W$ and satisfies $\sigma_u(x_i) = x_i^p$.
Using the isomorphisms $\iota_{\tilde{x}}, \iota^x$, we may now write $\varphi_u = b_u \sigma_u$ and $\varphi_{\tilde{x}} = b \sigma$, where the fiber of $b_u \in G(R^x[1/p])$ at $\tilde{x}$ is $b\in G(W[1/p])$. We also fix a choice of $G_W$-valued co-character $\mu$ that induces the filtration on $_{\cris}\V_{\tilde{x}}$) and on 
 $_{\cris}\V^x$ via $\iota$.)  

\begin{proposition}\label{prop frob filt}
   We have $b_u \in G(R^x)\sigma(\mu(p))$. 
\end{proposition}
\begin{proof}
Let $_{\cris}\V^x_i$ denote the weight spaces of $_{\cris}\V^x$ under $\mu$. Clearly, $\Fil^n = \ensuremath{\vcenter{\hbox{\scalebox{2}{$\oplus$}}}}_{i\geq n} {_{\cris}\V^x_i}$. The strong-divisibility condition yields that $\ensuremath{\vcenter{\hbox{\scalebox{2}{$\oplus$}}}}_i \frac{1}{p^i} b_u(\sigma({_{\cris}\V^x_i})) = {_{\cris}\V^x} $. Here, $\sigma$ denotes the semi-linear Frobenius on $W$, and acts as the identity on the coordinates $x_i$. Further, the action on ${_{\cris}\V^x}$ is via the $\Z_p$ structure given by Corollary \ref{cor:Gstructure}.

By definition, we have that $\mu(p)(_{\cris}\V^x_i) = p^i _{\cris}\V^x_i $. Therefore, we have that $_{\cris}\V^x = \ensuremath{\vcenter{\hbox{\scalebox{2}{$\oplus$}}}}_i b_u \sigma (\mu(p)^{-1} _{\cris}\V^x_i)$, whence $_{\cris}\V^x = b_u \sigma(\mu(p)^{-1}) ( {_{\cris}\V^x})$. It therefore follows that $b_u \sigma(\mu(p)^{-1}) \in \GL({_{\cris}\V^x})$. The proposition follows as $G(R^x) = G(R^x[1/p]) \cap \GL({_{\cris}\V^x_i})$.  
\end{proof}

We have the immediate corollary:
\begin{corollary}\label{cor frob filt}
    We have $b \in G(W) \sigma(\mu(p))$.
\end{corollary}

\subsubsection{The Fontaine-Laffaille module on complete local rings}\label{subsubsec: completelocalrings theorem}
We will now prove the main result of this section, which is to give a description of the Fontaine-Laffaille data on $\ShimK^x$ analogous to the case of abelian type Shimura varieties following Kisin \cite{Kisinintegral}. Let notation be as in the previous subsections. Recall that we have chosen an isomorphism $_{\cris}\V_{\tilde{x}} \otimes R^x \rightarrow _{\cris}\V^x$ of filtered modules compatible with the $G$-structure defined in Subsection \ref{subsection G-structure}. We note that such an isomorphism is unique only up to an element of $P(R^x)$, where $P$ is the parabolic subgroup of $G$ defined by the filtration. Also recall that we have chosen a $G$-valued co-character (defined over $W$) that splits $\Fildot$. Let $U^{\opp}$ denote the opposite unipotent of $\mu$, and let $\widehat{U}^{\opp}$ denote the completion of $U^{\opp}$ at the identity mod $p$. 

We now fix $\iota_{\tilde{x}}$. We note that the choice of $\iota$ fixes the choice of $\iota^x$.  
\begin{theorem}\label{thm: Kisin deformation}
The isomorphism of filtered modules $\iota: _{\cris}\V_{\tilde{x}}\otimes R^x \rightarrow _{\cris}\V^x$ can be chosen so that $b_u = u\cdot b$, where $u\in U^{\opp}(R^x)$. Further, the map induced by the point $u:\Spf R^x \rightarrow  \widehat{U}^{\opp}$ is an isomorphism. 
    
\end{theorem}

\begin{proof}
We will first prove that we may choose $\iota$ such that $u \in \widehat{U}^{\opp}$, and will then use the versality of the Kodaira Spencer map to show that the tautologocal map $u$ induces is an isomorphism. The first part of the proof will follow \cite[Proposition 4.7]{KLSS}. 

Let $I \subset R^x$ denote the kernel of $R^x \xrightarrow{\tilde{x}} W$. We have that $b_u = u'b$, where $u'\in G(R^x)$ with $u' \equiv \Id \bmod I$. Therefore, we also have that $u' \in U^{\opp} \bmod I$. We may write $u' = \lambda u$ with $\lambda \in P$ and $u\in U^{\opp}$, with both $\lambda$ and $u$ reducing to the identity modulo $I$. Replacing $\iota$ by $\lambda \circ \iota$ has the effect of replacing $b_u$ by $\lambda^{-1} u'b \sigma(\lambda) = u (b  \sigma(\lambda) b^{-1}) b $. Define $u'(1)$ to be $u (b  \sigma(\lambda) b^{-1})$. As $\sigma(I)$ is generated by $t_i^p$, we have that $u'(1) \in \widehat{U}^{\opp} \bmod I^p$, and we may thus write $u'(1) = \lambda(1) u(1)$, with $\lambda(1) \in P$ reducing to the identity modulo $I^p$ and $u(1)$ reducing to the identity modulo $I$. Iterating this process, we may change coordinates by an element of $P$ again and replace $b_u$ by $u'(2) b = u(1) (b \sigma(\lambda(1)) b^{-1}) b$, whence $u'(2) \in U^{\opp} \bmod I^{p^2}$. Iterating this process is equivalent to modifying $j$ by the convergent product $\hdots  \lambda(2) \cdot \lambda(1) \cdot \lambda$, and has the effect of writing $b_u = u \cdot b$ with $u\in U^{\opp}(R^x)$ as required. 

We will now prove that the tautological map $\Spf R^x \rightarrow \widehat{U}^{\opp}$ induced by $u$ is an isomorphism. The leading order term of the connection is $du $ (see \cite[Proof of Lemma 1.5.2]{Kisinintegral}\footnote{Note that in \emph{loc. cit.}, the leading order term is described as $u^{-1}\cdot du$. As $u\equiv \Id$ mod $I$, $du$ and $u^{-1}du$ have the same leading term.}). Therefore, the Kodaira spencer map reduces mod $I$ to the element $du_{\tilde{x}}: \Hom(T_{\tilde{x}} R^x \rightarrow \End(\Gr _{\cris}\V_{\tilde{x}}) )$, where $du_{\tilde{x}}$ is the restriction of $du$ to the tangent space of $T_{\tilde{x}}T^x$ of $R^x$ at $\tilde{x}$. The image of $du_{\tilde{x}}$ is contained in $\Lie U_{\tilde{x}}^{\opp}$. The versality of the Kodaira-Spencer map implies that $du_{\tilde{x}}$ is an immersion. As the dimension of $\widehat{U}^{\opp}$ is the same as the dimension of $R^x$, it follows that $u$ is an isomorphism. 



\end{proof}





\subsection{Versality of Filtrations}
 The setting and notation will be as in Section \ref{subsubsec: completelocalrings theorem}. Let $\tilde{x}'$ denote a $W$-valued lift of $x$. The $F$-crystals underlying the Fontaine-Laffaille modules $_{\cris}\V_{\tilde{x}}$ and $_{\cris}\V_{\tilde{x}'}$ are canonically isomorphic, and indeed are canonically isomorphic to the F-crystal $_{\cris}\V_x$. This canonical isomorphism is obtained by considering the horizontal continuation $s_{\nabla}$ of elements $s \in _{\cris}\V_{\tilde{x}}$ to $_{\cris}\V^x$ with respect to the connection $\nabla$ on $_{\cris}\V^x$. We will use this identification to view $\Fildot(_{\cris}\V_{\tilde{x}})$ and $\Fildot(_{\cris}\V_{\tilde{x}'})$ as filtrations on $_{\cris}\V_x$, i.e. different filtrations on the same underlying $F$-crystal. The content of the following proposition is that points $\tilde{x}\in R^x(W)$ are in bijection with filtrations of the form $g \cdot \Fildot(_{\cris}\V_{\tilde{x}})$, where $g\in \widehat{U}^{\opp}(W)$. This result will be used crucially in a later section to define the canonical lift of ($\mu$-)ordinary  and then prove our CM lifting theorems. 

\begin{proposition}\label{filts}
There is a bijection between $\widehat{U}^{\opp}(W)$ and points $\tilde{x}' \in R^x(W)$, where the bijection assigns to $g \in \widehat{U}^{\opp}(W)$ the Fontaine-Laffaille module with underlying $F$-crystal $_{\cris}\V_x$ and filtration $g\cdot \Fildot(_{\cris}\V^{\tilde{x}})$.
 \end{proposition}


We will first need the following result about flat sections on $_{\cris}\V^{x}$. Let $\cO$ denote the ring of rigid-analytic functions on the rigid-analytic space given by the tube of $\integralShimK$ over $x$. The tube is just a $\ShimK$-dimensional open unit ball around $\tilde{x}$. Note that $R^x[1/p] \subset \cO$ is the subset of bounded functions. Recall that we have the coordinates $x_i \in R_x$ whose vanishing defines $\tilde{x}$. Let $y_i \in \cO$ be defined by $y_i = \frac{x_i}{p}$. We have the inclusion $\cO \subset W[\frac{1}{p}]\langle y_i \rangle$, given by restricting functions on the open unit ball to the closed ball with radius $\frac{1}{p}$. We let $W\langle y_i \rangle$ denote functions on this closed ball that are bounded above by 1. We are now ready to state the following lemma.
\begin{lemma}\label{lem:flat sections on connection}
We have the following results: 
\begin{enumerate}
    \item The connection $\nabla$ has a basis of solutions over $\cO$. Further, $(_{\cris}\V^x,\nabla)|_{W\langle y_i\rangle}$ is isomorphic to $(_{\cris}\V^x_{\nabla = 0} \otimes_W W\langle y_i\rangle, d)$ . Consequently, the horizontal continuation of every $s\in _{\cris}\V_{\tilde{x}}$ is an element of $_{\cris}\V^x \otimes_{R^x} W\langle y_i \rangle $. 

    \item Let $s\in _{\cris}\V_{\tilde{x}}$ be some element. Let $\tilde{s}$ denote the element $s\otimes 1 \in {_{\cris}\V^x}$ under the isomorphism in Theorem \ref{thm: Kisin deformation}. Let $s_{\nabla}$ denote the horizontal continuation of $s$ to $_{\cris}\V^x \otimes \cO$. There exists a unique element $g\in G(\cO)$ such that $g\cdot \tilde{s} = s_{\nabla}$ where $g$ is independent of $s$. Further, $g\in G(W\langle y_i\rangle )$.

    \item $g \equiv u^{-1} \mod (x_i)^2$, where $u$ is the tautological point $u\in \widehat{U}^{\opp}(R^x)$ as in Theorem \ref{thm: Kisin deformation}. Further, the coefficients of the series expansion of $g$ in the coordinates $y_i$ are divisible by $p^2$ for the non-linear and non-constant terms.

\end{enumerate}

\end{lemma}

\begin{proof}
This lemma is proved by direct calculation. First note that $(_{\cris}\V^x,\nabla)$ has a full set of solutions over $W[1/p][[x_i]]$, with flat sections given by: \begin{equation}\label{eq: flat sections}
    \displaystyle s_{\nabla} = \sum_{\vec{w}} \prod_i \frac{(-x_i)^{w_i}}{i!} \prod_i \nabla(\frac{d}{dx_i})^{w_i}(\tilde{s}),
\end{equation} 
where the sum is taken over all tuples $\vec{w} = (w_i)$ of non-negative integers. 
    \begin{enumerate}
        \item That $\nabla$ has a basis of solutions over $\cO$ follows from the Frobenius structure (Dwork's trick). The fact that the flat bundle $ (_{\cris}\V^{x},\nabla)|_{W\langle y_i \rangle}$ is trivial is a general fact about vector bundles with flat connection on $W[[x_i]]$ pulled back to $W\langle y_i \rangle$ for primes $p > 2$. Indeed, the connection matrix has entries valued in $W[[x_i]]$, and after substituting $x_i = py_i $, it is evident that the resulting expressions are contained in $W\langle y_i\rangle$, and indeed, we have $s_{\nabla}|_{W\langle y_i \rangle} \equiv \tilde{s}|_{W\langle y_i \rangle} \bmod p$. 

        \item By \cite[Lemma 1.5.2]{Kisinintegral} and \cite[E.1]{Kisinmodppoints}\footnote{As the erratum indicates, the conclusion that $\nabla \in \protect\mathrm{Lie} U \otimes \Omega^1_{R^x}$ is false without extra assumptions (for example, the point $x$ being ordinary), but the weaker claim $\nabla \in \mathrm{Lie} G \otimes \Omega^1_{R^x}$ is true.}, it follows that $\nabla \in \mathrm{Lie} G \otimes \Omega^1_{R^x}$. Therefore, $g$ is indeed a point of $G$, a-priori valued in $W[1/p][[x_i]]$. The fact that it is an $\cO$-valued point (and indeed, a $W\langle y_i\rangle$-valued point ) of $G$ follows from the first part (and the fact that $G(W\langle y_i \rangle) = \GL(W\langle y_i \rangle) \cap G(W[1/p][[x_i]])$). 

        \item This follows from Equation \eqref{eq: flat sections} and the fact that the leading term of $\nabla$ is $du$.
    \end{enumerate}
\end{proof}

With this in hand, we will now prove Proposition \ref{filts}:
\begin{proof}
    Let $h = g^{-1}$. By Lemma \ref{lem:flat sections on connection} (3), we have that $h \equiv u \bmod (y_i)^2\cap (p^2)$. We have that $\tilde{s} = h \cdot s_{\nabla}$. 

    We let $\Fildot$ denote the filtration $\Fildot(_{\cris}\V_{\tilde{x}})$. The filtration on $_{\cris}\V^x$ is given by $\widetilde{\Fildot} = \Fildot \otimes_W R^x$. Let $\Fildot_{\nabla}$ denote the horizontal continuation of $\Fildot$. We then have that $h\cdot \Fildot_{\nabla} = \widetilde{\Fildot}$.  Let $P_{\Fildot_{\nabla}} \subset G$ be the parabolic subgroup of $G$ stabilizing the filtration $\Fildot_{\nabla}$. We may write $h = u' \cdot \lambda'$, where $u' \in U^{\opp}(W\langle y_i \rangle)$ and $\lambda' \in P_{\Fildot_{\nabla}}(W\langle y_i \rangle)$ -- we have that $\widetilde{\Fildot} = u' \cdot \Fildot_{\nabla}$. Note that $u'\equiv h\mod (y_i)^2\cap (p^2)$.

    Let $U^{\opp}_{1/p} \subset \widehat{U}^{\opp,\textrm{rig}}$ denote the closed ball of radius $\frac{1}{p}$ around the identity. The element $u'$ induces a rigid-analytic map $ \Sp W\langle y_i \rangle \rightarrow \widehat{U}^{\opp,\textrm{rig}}$, and the conditions on $u'$ imply that the image is contained in $U^{\opp}_{1/p}$. Note that unipotent groups are uniquely $p$-divisible, and so the map ``dividing by $p$'' gives an isomorphism $U^{\opp}_{1/p} \rightarrow U^{\opp,\textrm{rig}}$. Here, we abuse notation and also allow $U^{\opp}$ to denote the formal scheme obtained by $p$-adic completing $U^{\opp}$. Therefore, by composing $u'$ by the ``dividing by $p$'' map, we obtain a map $u'': \Sp W[1/p]\langle y_i \rangle \rightarrow U^{\opp,\textrm{rig}}$. Both these rigid spaces admit natural integral models -- $\Spf W \langle y_i\rangle$ and $ U^{\opp}$. The congruence conditions on $u'$ imply that $u''$ is induced by a map on these integral models. It suffices to prove that $u''$ is an isomorphism, which would follow from the claim mod $p$. The congruences established in Lemma \ref{lem:flat sections on connection} together with Theorem \ref{thm: Kisin deformation} imply that $u''$ is an isomorphism modulo $(y_i)^2$, and that the coefficients in the series expansion of $u''$ in the coordinates $y_i$ are divisible by $p$ for the non-linear and non-constant terms. It is now clear $u''$ is an affine linear map mod $p$, and it suffices to check that it is an isomorphism at the tangent space of some closed point. But this follows from the fact that $u''$ is an isomorphism mod $(y_i)^2$. The proposition follows from here. 
    
    
\end{proof}

\section{The $\mu$-ordinary locus: CM lifts and Tate's isogeny theorem}\label{sec: CM stuff}
In this section, we will define the notion of $\mu$-ordinariness. This notion generalizes the notion of ordinary abelian varieties. We will then prove that the $\mu$-ordinary locus is open-dense, and that every $\mu$-ordinary point admits a lift to a special point. This generalizes the Serre-Tate canonical lift of an ordinary abelian variety (and generalizes work of Moonen for Shimura varieties of PEL type, and work of \cite{ShankarZhou} in the case of Hodge type Shimura varieties). 

We keep the notation of the previous section.
 Let $B \subset G$ be a Borel subgroup defined over $\Z_p$, and let $T\subset B\subset G$ be a maximal torus that splits over an etale extension of $\Z_p$(that $G/\Z_p$ is reductive implies that these objects exist). Let $x \in \integralShimK(\F)$. Associated to the $F$-crystal $_{\cris}\V_x$ is the Frobenius element $b_x\in G(L)$, which is well defined up to $\sigma$-conjugation by $G(W)$. Let $\nu_x \in X_*(T)$ be the dominant element that is conjugate to $\nu_{b_x}$, the Newton cocharacter associated to $b_x$ (see work of Kottwitz \cite{Kott1} and \cite{Kott2} for definitions and more details). Note that the conjugacy class of $\nu_{b_x}$ depends only on the $\sigma$-conjugacy class of $b_x$, and so the definition of $\nu_x$ is independent of the choice of representative of the $\sigma$-conjugacy class of $b_x$. We note that $\nu_x$ is defined over $\Z_p$. Let $\mu_x\in X_*(T)$ denote the dominant co-character with $b_x \in G(W) \sigma(\mu_x)(p) G(W)$. The co-character $\mu_x$ might not be defined over $\Z_p$, and let $\bar{\mu}_x = \frac{1}{[E_p:\Q_p]} \sum_{\tau \in \Gal(E_p/\Q_p)}\tau(\mu_x)$ denote the Galois average of $\mu$, where $E_p/\Q_p$ is some finite unramified extension splitting $T$. We have that $\nu_x \preceq \bar{\mu_x}$ (for example, see \cite{Gashi}), where two dominant fractional cocharacters satisfy $\mu' \preceq \mu''$ if and only if $\mu'' - \mu'$ can be written as a non-negative rational linear combination of positive co-roots. We have the following definition. 

  \begin{definition}\label{def:ord}
      \hfill \begin{enumerate}
          \item The point $x$ is said to be $\mu$-ordinary if $\nu_x = \bar{\mu}_x$. 
          \item The point $x$ is said to be ordinary if it is $\mu$-ordinary and if $\mu_x$ is defined over $\Z_p$. 
      \end{enumerate}
  \end{definition}
We remark that when $G=T$ is itself a torus, then every point is $\mu$-ordinary. 
The notion of $\mu$-ordinariness was first introduced by Rapoport in his ICM talk (see also \cite{Wortman}, \cite{MoonenSerreTate}, and \cite{ShankarZhou}). The $\mu$-ordinary locus will be open-dense if nonempty. This was done in \cite{Wortman} in the case of Shimura varieties of abelian type, when the level structure at $p$ is hyperspercial (see also \cite{Bultel}). We prove this in the case when $p$ is a large enough prime -- we remark that our proof works both in the abelian and exceptional cases. 

\begin{theorem}\label{nonemptinessofordinary}
    Let $p$ be a large enough prime and let $v\mid p$ be a prime of $E$. Then the $\mu$-ordinary locus of $\integralShimK_{\F_v}$ is non-empty.
\end{theorem}
We will prove this result at the end of this section. 
The content of \cite[Proposition 7.2]{Wortman} is the following result: 
\begin{proposition}[Wortman]\label{prop:frobmuord}
    Let $x \in \integralShimK(\F)$ be a $\mu$-ordinary point with $ b_x, \nu_x$ and $\mu_x$ as above. Then there exists $g\in G(W)$ such that $g^{-1}b_x\sigma(g) = (\sigma \mu_x)(p)$. 
\end{proposition}

We now change coordinates on (by $\sigma$-conjugating by an element of $G(W)$) and we may assume that $b_x = \sigma(\mu_x) p$.  The following proposition defines the canonical lift of a $\mu$-ordinary point. 
\begin{proposition}\label{prop: canonical lift}
    Consider the Fontaine-Laffaille module with underlying $F$-crystal $_{\cris}\V_x$, and with the filtration defined by the co-character $\mu_x$. This corresponds to a (necessarily) unique point of $\ShimK^x$. 
\end{proposition}
\begin{proof}
    This result follows almost immediately from Proposition \ref{prop frob filt}. Indeed, let $\tilde{x}$ denote some $W$-valued lift of $x$, and let $\mu$ denote a co-character inducing the filtration on $_{\cris}\V_{\tilde{x}}$. The Fontaine-Laffaille conditions imply that filtrations induced by $\mu_x$ and $\mu$ agree on $_{\cris}\V_{\tilde{x}} \bmod p$. 

    Consider the parabolic subgroups of $P_x\subset G$ and $P_{\mu} \subset G$ induced by $\mu_x$ and $\mu$ respectively. That the filtrations agree modulo $p$ implies that $P_x = P_{\mu} \bmod p$. By \cite[Corollary 5.2.7]{Conrad:ReductiveSchemes}, the two parabolics are conjugate over some etale extension of $W$, and therefore over $W$ itself. As the two parabolics agree modulo $p$, they are conjugate by an element $g\in \hat{U}^{\opp}(W)$ (with $\hat{U}^{\opp}$ the opposite unipotent of $\mu$). Therefore, $\Fil_x = g \cdot \Fil_{\mu}$, and so the proposition follows by Proposition \ref{prop frob filt}. 
    
\end{proof}

\begin{definition}\label{def: canonical lift}
Let $x\in \Shim(\F)$ be a $\mu$-ordinary point. We let $x^{\can}$ denote $W$-valued lift constructed in Proposition \ref{prop: canonical lift}, and call it canonical lift of $x$. 
\end{definition}
In the ordinary case, this construction directly generalizes the Serre-Tate canonical lift of an ordinary abelian variety. More generally, this generalizes the notion of canonical lifts of $\mu$-ordinary points of PEL-type Shimura varieties defined by Moonen in \cite{MoonenSerreTate}, and of $\mu$-ordinary points of Hodge-type Shimura varieties defined in \cite{ShankarZhou}. Analogous to those cases, we will shortly prove that $x^{\can}$ is actually a CM point. For now, we remark that the canonical lift is functorial in prime-to-$p$ Hecke correspondence. Indeed, prime-to-$p$ Hecke correspondences are finite etale correspondences on the integral model, and therefore every $\tilde{y}\in \tau(\tilde{x})$ is a $W$-valued point of $\integralShimK$ whenever $\tilde{x}$ is. Further, the action of prime-to-$p$ Hecke correspondences on $W$-valued points preserves the isomorphism class of the Fontaine-Laffaille modules attached to these points. Therefore, we have the immediate corollary:

\begin{corollary}\label{prop: functoriality of canonical lift}
    Let $x^{\can}$ denote the canonical lift of a $\mu$-ordinary point $x$, and let $\tilde{y} \in \tau(x^{\can})$ where $\tau$ is a prime-to-$p$ Hecke correspondence. Let $y$ be the specialization of $\tilde{y}$. Then, $\tilde{y} = y^{\can}$. In particular, $\tilde{y} = x^{\can}$ if $y = x$. 
\end{corollary}

We can now apply this in conjunction with Proposition \ref{prop:charecterization of special points} and Corollary \ref{cor:everypointCM} to obtain the following result:
\begin{theorem}\label{CMlifts}
    Every $\mu$-ordinary point lifts to a characteristic-zero special point.
\end{theorem}
We obtain the following corollary. 
\begin{corollary}\label{cor: crystalline frob ss}
    Let $x \in \integralShimK(\F_q)$ be $\mu$-ordinary, with $q = p^n$, and let $_{\cris}\varphi_x$ be the crystalline (semilinear) Frobenius endomorphism on $\D_x$. Then, the linear map $_{\cris}\varphi_x^n$ is semisimple. 
\end{corollary}

\subsection{Non-emptiness of the $\mu$-ordinary locus}
We will now prove that the $\mu$-ordinary locus is non-empty at all large enough primes $v$ of the reflex field $E$, by constructing zero-dimensional Shimura varieties that specialize mod $v$ to $\mu$-ordinary points. It suffices to treat the case of adjoint groups, so we assume throughout that $G$ is adjoint. Let $(T,h) \rightarrow (G,X)$ denote an embedding of Shimura data such that the image of $T$ is a maximal torus of $G$. Let $E_T$ denote the reflex field of $(T,h)$ -- note that this is just the field of definition of the Hodge co-character $\mu_h$ associated to $h$. Let $w$ be a place of $E_T$ dividing the place $v$ of $E$. Let $p$ denote the rational prime that $v$ divides. 
We make the following definition. 

\begin{definition}\label{def: quasisplit torus}
    Let $T'\subset G'$ denote a maximal torus inside a reductive group defined over $\Z_p$. We say that $T'$ is \emph{quasi-split} if $T'$ is contained in a Borel subgroup $B\subset G$ defined over $\Z_p$. 
\end{definition}
We remark that quasi-split tori always exist, they are all conjugate to each other over $\Z_p$, and a torus inside a $\Z_p$-reductive group is quasi-split if and only if it contains a maximal split torus. 

\begin{proposition}\label{prop nonemptiness ofmuord: first redution}
The zero-dimensional Shimura variety modulo $w$ induces a $\mu$-ordinary of $\integralShimK_{\F_v}$ if the following two conditions are satisfied:

\begin{enumerate}
    \item $T$ is quasi-split at $p$.
    \item The co-character $\mu_h \in X_*(T_{E_{T,w}})$ is dominant with respect to a Borel subgroup of $G$ defined over $\Z_p$.
\end{enumerate}
\end{proposition}
\begin{proof}
As remarked earlier, the mod $w$ zero-dimensional Shimura variety is $\mu$-ordinary for the group $T$. Therefore, we have that the Newton co-character is the Galois average of $\mu_h$. However, in order to show that we also have $\mu$-ordinariness for the group $G$, we must show the same equality but after replacing both the Newton and the Hodge co-character with their dominant conjugates (dominant with respect to some Borel subgroup defined over $\Q_p$). However, $\mu_h$ is dominant by definition, and as $B$ is defined over $\Q_p$, so is its Galois average. The result follows. 
\end{proof}

The absolute Galois group of $\Q$ acts on $X_*(T)$. We recall Chai and Oort's definition from \cite{ChaiOortJacobians}.

\begin{definition}
    We say that $(T,h)$ is a Weyl special point if the image of the Galois action $\Gal(\overline{\Q}/\Q)$ on the co-character lattice $X_*(T)$ contains the Weyl group.
\end{definition}
Given this setup, let $E(G)$ denote the finite extension of $\Q$ such that the image of the Galois action $\Gal(\overline{E(G)}/E(G))$ equals the Weyl group. Let $E'_T$ denote the splitting field of $T$. Note that $E(G)$ contains $E$ and $E'_T$ contains $E_T$.

\begin{proposition}\label{prop nonemptinessofmuord second reduction}
    Let $(T,h)$ be as above. Suppose that $T$ is quasi-split at $p$, and that $(T,h)$ is a Weyl special point. Then there exists a place $w$ of $E_{T}$ that satisfies the second condition of Proposition \ref{prop nonemptiness ofmuord: first redution}. 
\end{proposition}
\begin{proof}
    It suffices to replace $E$, and $E_T$ with $E(G)$, and $E'_T$ respectively. Abusing notation by letting $v$ denote a place of $E(G)$ dividing the original place of $E$, we must find a place $w\mid v$ of $E'_T$ such that $\mu_h \in X_*(T_{E'_{T,w}})$ satisfies the positivity condition with respect to a $\Z_p$-Borel.  Given a co-character $\mu$ defined over $E'_T$ and a place $w\mid v$, we let $\mu^w$ denote the base-change of $\mu$ to $E_{T,w}$. We also fix a choice of Borel subgroup over $\Z_p$ containing $T$.

    We now pick a place $w'\mid v$.
    By construction, the Weyl-group orbit of $\mu_h$ is defined over $E'_T$ and therefore over $E'_{T,w'}$. As $T$ is split over $E'_{T}$, we have that $\mu^{w'}_h$ and any other element in its Weyl orbit are conjugate over $E'_{T,w'}$. Let $g$ be the Weyl group element satisfying $g\cdot \mu^{w'}_h$ is dominant with respect to $B$. Let $\tau \in \Gal(E'_T/E(G))$ satisfy $\tau(\mu_h) = g\cdot \mu_h$. Then by construction, $\tau(\mu_h)^{w'} \in X_*(T_{E'_{T,w'}})$ is dominant. But this is equivalent to $\mu_h^{\tau^{-1}(w')} \in X_*(T_{E'_{T, \tau^{-1}(w')}})$ being dominant as $B$ is defined over $\Q_p$. Therefore, the proposition follows by picking $w = \tau^{-1}(w')$.
    
\end{proof}

We are now ready to prove Theorem \ref{nonemptinessofordinary}.

\begin{proof}
    By Propositions \ref{prop nonemptiness ofmuord: first redution} and \ref{prop nonemptinessofmuord second reduction}, it suffices to find a Weyl special point $(T,h)$ where $T$ is quasi-split at $p$. But this follows directly from \cite[Proposition 5.11]{ChaiOortJacobians} by choosing $\Sigma_1 = \{p,\infty \}$, $U_{\infty}$ as in Remark 5.12 of \emph{loc. cit.}, and $U_p$ to be the set of set of elements in $\Lie G(\Z_p)$ which reduce to an element of $\Lie G(\Z_p)$ whose stabilizer in $G_{\F_p}$ is equal to a maximal torus containing a maximal split torus. 
\end{proof}

\subsection{An analogue of Tate's isogeny theorem for ordinary points}
We will now prove prove an analogue of Tate's isogeny theorem for abelian varieties in the setting of ordinary points. 

\begin{theorem}\label{thm: Tate isogeny}
    Let $x,y\in \integralShimK(\F_q)$ be ordinary points such that the $\ell$-adic Frobenii $\varphi_x$ and $\varphi_y$ are conjugate in $G(\Q_{\ell})$. Then, the rational Hodge structures associated to the canonical lifts $x^{\can}$ and $y^{\can}$ are isomorphic.
\end{theorem}
In the setting of (ordinary) abelian varieties, this result would translate to abelian varieties over finite fields being isogenous if their $\ell$-adic Frobenii are conjugate.

\begin{proof} 
    Let $(T_x,\mu_x)$ and $(T_y,\mu_y)$ denote the zero-dimensional Shimura data inducing the special points $x^{\can}$ and $y^{\can}$. Let $E_x$ denote the reflex field of $(T_x,\mu_x)$ and let $v$ be the place of $E_x$ such that $x 
 = x^{\can} \bmod v$. Recall that we have fixed $V$, a $\Q$-representation of $G$, which we will now consider as a representation of $T_x$. By \cite{Milnefinitefieldmotives}, there exists a surjective morphism of Shimura data $(T,\mu) \rightarrow (T_x,\mu_x)$ where $(T,\mu)$ is a zero-dimensional Shimura datum of Hodge type. Let $E$ denote the reflex field of $(T,\mu)$ ($E$ contains $E_x$), and let $v'\mid v$ be some place of $E$ above $E_x$. Let $W$ be the tautological symplectic representation of $T$ associated to a Hodge embedding. Let $z \in S(T,\mu)$ be a point that maps to $\tilde{x}$ and let $A$ denote the CM abelian variety associated to $z$. Note that $T$ is the Mumford-Tate group of $A$. Let $\beta \in \End(A)$ denote the endomorphism that specializes to the Frobenius endomorphism of $A \bmod v'$. We will abuse notation and let $\beta \in T(\Q)$ also denote the Betti-realization of the endomorphism $\beta$. Let $_{\et}\beta_p$ and $_{\dR}\beta$ denote the endomorphisms of $_{\et}V_p$ and $_{\dR}V$ induced by $\beta$. Restricting the Galois representation to the local Galois group at $v$, we may apply the Crystalline functor $\D_{\cris}$ to $_{\et}\beta$ and obtain an endomorphism of $_{\cris}V,\Fildot$ -- this endomorphism is just the power of the crystalline Frobenius operator on $_{\cris}V ,\Fildot$, and equals $_{\dR}\beta$ under the crystalline-deRham comparison isomorphism. All of this follows from the fact that $\beta \in T(\Q)$ is induced by an endomorphism of the abelian variety $A$. 
 
 We set $q = p^n$. Let $\alpha\in T_x(\Q)$ denote the image of $\beta$. As $V$ is in the Tannakian category generated by $W$ (as $T$-representations), we have that the same compatibility holds for the various realizations of $\alpha$. That is, we have that $_{\et}\alpha_{\ell}$ acts as Frobenius $_{\et}\varphi_{x,\ell}$ on $_{\et}V_{x,\ell}$ for $\ell \neq p$, and the Crystalline functor $\D_{\cris}$ applied to $_{\et}\alpha_p$ the $n$th power of the Crystalline Frobenius $_{\cris}\varphi_x$ on $_{\cris}V_{x}$, and therefore that $(_{\cris} \varphi_x)^n = _{\dR}\alpha$ under the Crystalline-deRham comparison isomorphism. We will now prove that the group generated by $\alpha$ is Zariski-dense in $T_x$. Let $L$ denote any field containing $E_x$. Let $_{\dR}T_{\alpha} \subset \GL(_{\dR}V_{\tilde{x},L})$ denote the torus\footnote{By replacing $q$ by a sufficiently large power, this group will always be connected and therefore a torus.} defined by taking the Zariski-closure of the group generated by $_{\dR}\alpha$. It suffices to prove that the Hodge co-character factors through $_{\dR}T_{\alpha}$ (note that this doesn't depend on the field $L$). We pick $L = W(\F_q)[1/p]$. By the Crystalline-de Rham comparison (and the compatibility outlined above), we may replace $_{\dR}\alpha$ and $_{\dR}V_{x^{\can},L},\Fildot$ by $_{\cris}\varphi^n$ and $_{\cris}V_{\tilde{x}},\Fildot$. As the point $x$ is ordinary, we have a direct sum decomposition of the crystal $_{\cris}V_x,_{\cris}\varphi_x = \bigoplus _{\cris}V_x^i , \oplus _{\cris}\varphi^i$, where $\frac{_{\cris}\varphi^i}{p^i}$ is a bijective $\sigma$-linear Frobenius morphism. By definition, the Hodge filtration on $_{\cris}V_x$ defined by the canonical lifting is $\Fil^j = \bigoplus_{i\leq j}\ _{\cris}V_x^i$. We have that $_{\dR}\alpha$ is just the $n$th iterate of $_{\cris}\varphi_x$, and the eigenvalues of $_{\dR}\alpha$ on $_{\cris}V_x^i$ have $p$-adic valuation exactly equal to $ni$. In order to prove that the Hodge co-character factors through $T_{\alpha}$ it suffices to prove the following general claim:
\begin{claim}

 Let $K$ be a finite extension of $\Q_p$, let $T'$ be a and let $V'$ be a faithful algebraic representation of $T'$ defined over $K$. Let $\mu: \G_m\rightarrow T'$ be a co-character that induces weight decomposition $V' = \bigoplus V^{\prime,i}$. Let $\alpha' \in T'(K)$ be an element that acts on $V^{\prime,i}$ with eigenvalues having valuation $ni$. Then, $\mu$ factors through $T'_{\alpha'}$ -  the Zariski closure of the group generated by $\alpha'$.
    
\end{claim}
\begin{proof}
By replacing $K$ by a finite extension if need be, we may assume that $T'$ is a split torus. Without loss of generality, we may also assume that the rank of $T'$ equals the dimension of $V'$. Let $\chi$ denote any character of $T'$ such that $\chi(\alpha') = 1$. It suffices to prove that $\chi\circ \mu$ is the trivial character on $\G_m$. We may write $T' = \prod \prod T^{\prime,i}$ according to the decomposition $V' = \bigoplus V^{\prime,i}$. Picking coordinates on $V^{\prime,i}$, we may write $T^{\prime,i} = \prod \G_m$, with the diagonal $\G_m$ acting by scalars on $V^{\prime,i}$. In these coordinates, $\alpha'$ will decompose as $(\alpha^{\prime,i})$ with the entries of $\alpha^{\prime,i}$ having $p$-adic valuation $ni$. Let $\chi$ be a character on $T'$. Then we may write $\chi$ as $\prod_i \chi_i$, with $\chi_i = (a_{i,1},\hdots a_{i,N_i})$, where $\chi_i((x_j)) = \prod x_j^{a_{i,j}}$ for $(x_j) \in T^{\prime,i} = \prod \G_m$. A necessary condition for $\chi(\alpha') = 1$ is that the $p$-adic valuation of $\chi(\alpha')$ is 0, i.e.  $\sum_i ni(a_{i,1} + \hdots a_{i,N_i}) = 0$. On the other hand, the co-character $\mu$ is given by $\mu = \prod \mu_i$, with $\mu_i(z) = (z^i,z^i,\hdots z^i)$. Therefore, $\chi \circ \mu$ is trivial if and only if $\sum_i i(a_{i,1} + \hdots a_{i,N_i}) =0$. Therefore, we've shown that if $\chi(\alpha) = 1$, then $\chi \circ \mu$ is trivial, and therefore $\mu$ does factor through $T_{\alpha'}$ as required.

\end{proof}
Therefore, we have that $T_{\alpha} = T_x$ as required. We have the analogous statement for the Frobenius element associated to $y$ and the Torus $T_y$. By assumption, the Frobenii satisfy the same characteristic polynomial on $V$. Therefore, there exists a $\Q$-endomorphism $X \in \GL(V)$ with $X^{-1}\alpha X = \gamma$. As the groups generated by $\alpha$ and $\gamma$ are respectively Zariski dense in $T_x$ and $T_y$ respectively, we have that $X^{-1} T_x X = T_y$. Finally, the Hodge co-characters are conjugate to each other because the CM-types are uniquely determined by the elements $\alpha$ and $\gamma$. The proposition follows. 
\end{proof}

\bibliography{biblio}
\bibliographystyle{plain}

\end{document}